\documentclass[11pt,reqno]{amsart}

\usepackage{amsmath,amssymb,mathrsfs,amsfonts,mathabx}
\usepackage{graphicx,cite,cases}
\usepackage{xcolor}
\newtheorem{theorem}{Theorem}[section]

\newtheorem{lemma}[theorem]{Lemma}

\numberwithin{equation}{section}

\begin{document}

\title[an adaptive finite element DtN method]{An adaptive finite element DtN method for the elastic wave scattering problem in three dimensions}

\author{Gang Bao}
\address{School of Mathematical Sciences, Zhejiang University, Hangzhou 310027, China.}
\email{baog@zju.edu.cn}

\author{Peijun Li}
\address{Department of Mathematics, Purdue University, West Lafayette, Indiana 47907, USA.}
\email{lipeijun@math.purdue.edu}

\author{Xiaokai Yuan}
\address{School of Mathematical Sciences, Zhejiang University, Hangzhou 310027, China.}
\email{yuan170@zju.edu.cn}

\thanks{The work of GB is supported in part by an NSFC Innovative Group Fund (No.11621101). The research of PL is supported in part by the NSF grant DMS-1912704.}



\begin{abstract}
Consider the elastic scattering of an incident wave by a rigid obstacle in three dimensions, which is formulated as an exterior problem for the Navier equation. By constructing a Dirichlet-to-Neumann (DtN) operator and introducing a transparent boundary condition, the scattering problem is reduced equivalently to a boundary value problem in a bounded domain. The discrete 
problem with the truncated DtN operator is solved by using the a posteriori error estimate based adaptive finite element method. The estimate takes account of both the finite element approximation error and the truncation error of the DtN operator, where the latter is shown to converge exponentially with respect to the truncation parameter. Moreover, the generalized Woodbury matrix identity is utilized to solve the resulting linear system efficiently. Numerical experiments are presented to demonstrate the superior performance of the proposed method.
\end{abstract}

\maketitle

\section{Introduction}

As a basic problem in classical scattering theory, the obstacle scattering problem refers to as the scattering of a time-harmonic wave by an impenetrable medium of compact support. It plays a fundamental role in diverse scientific areas such as radar and sonar, geophysical exploration, medical imaging, and nondestructive testing. The obstacle scattering problems have been extensively investigated in both the engineering and mathematical communities. A great number of numerical and mathematical results are available, especially for acoustic and electromagnetic waves \cite{CK-83, M-03, N-01}. Recently, the scattering problems for elastic waves have received ever-increasing attention due to the significant applications in geophysics, 
seismology, and elastography \cite{BHSY-jmpa18, BPT-mc10, CXZ-mc16, LWWZ-ip16}. For example, in medical diagnostics, by mapping the elastic properties and stiffness of soft tissues, they are able to give diagnostic information about the presence or status of disease. Compared with acoustic and electromagnetic waves, the scattering problems for elastic waves remain many issues on theoretical analysis and numerical computation because of the complexity of the model equation \cite{C-88, LL-86}. 

The elastic obstacle scattering problem is imposed in an unbounded domain, which needs to be truncated into a bounded domain in practice. Therefore, an appropriate boundary condition is required on the boundary of the truncated domain so that no artificial wave reflection occurs. Such a boundary condition is called a non-reflecting boundary condition or transparent boundary condition (TBC). Despite the large amount of work done so far, it is still one of the important and active research subjects on developing effective non-reflecting boundary conditions in the area of computational wave propagation \cite{BT-cpam80, EM-mc77, GG-wm03, GK-wm90, GK-jcp95, MK-jcp04, T-an99}. In this work, we construct a Dirichlet-to-Neumann (DtN) operator and develop a TBC for solving the elastic obstacle scattering problem in three dimensions. Based on the Helmholtz decomposition, the scattered field of the elastic displacement is split into the compressional and shear wave components which satisfy the Helmholtz equation and 
the Maxwell equation, respectively. Therefore, the DtN operator for the elastic wave equation can be obtained from the well-studied DtN operators for the Helmholtz and Maxwell equations. Since the TBC is exact, the artificial boundary could be put as close as possible to the obstacle in order to reduce the computational complexity \cite{JLZ-cicp13, LZZ-csam20}.  

To design an efficient numerical method, there are two more issues which need to be considered. The first issue concerns the truncation of the DtN operator. The nonlocal DtN operator is given as an infinite series, which has to be truncated into a sum of finitely many terms in actual computation. However, it is known that the convergence of the truncated DtN operator could be arbitrarily slow in the operator norm \cite{HN-2011-jcam}. From the computational viewpoint, it is important to answer 
the question how many terms are required in the summation in order to maintain a certain level of accuracy. Second, the solution may have local singularity when the obstacle has edges. The mesh should be fine around the nonsmooth part of the obstacle in order to capture the singularity of the solution; while the mesh could be coarse in other part of the domain where the solution is smooth. Hence, it is crucial to design an algorithm for mesh modification which can distribute equally the computational effort and optimize the computation. 

In this paper, we propose an adaptive finite element method with the truncated DtN operator to overcome the two difficulties mentioned above. Specifically, we consider the scattering of a plane wave by an elastically rigid obstacle in three dimensions. The exterior domain is assumed to be filled with a homogeneous and isotropic elastic medium. The elastic wave propagation is governed by the Navier equation. Based on the TBC, the exterior scattering problem is formulated equivalently into a boundary value problem in a bounded domain. The discrete problem is solved by using the finite element method with the truncated DtN operator. Based on the Helmholtz decomposition, a new duality argument is developed to obtain an a posteriori error estimate  between the solution  of the original scattering problem and the discrete problem. The a posteriori error estimate takes account of both the finite element approximation error and the truncation error of the DtN operator, where the latter is shown to decay 
exponentially with respect to the  truncation parameter. The estimate is used to design the adaptive finite element algorithm to choose elements for refinements and to determine the truncation parameter. The stiffness matrix is made of a sparse, real and symmetric matrix, which comes from the discretization of the variational formulation in the interior of the domain, and a dense low rank matrix given by vector products, which arises from the nonlocal TBC. The generalized Woodbury matrix identity is utilized  to solve the resulting linear system efficiently. Numerical experiments are presented to demonstrate the superior performance of the proposed method. 

Recently, the adaptive finite element DtN method has been developed to solve many acoustic and electromagnetic scattering problems, such as the obstacle scattering problems \cite{bzh-2020, jllz-2017-ccp}, the diffraction grating problems \cite{WBLLW-sinum15, JLWWZ-18}, and the open cavity scattering problem \cite{YBL-2020}. This paper is a non-trivial extension of our previous work on the two-dimensional elastic obstacle scattering problem \cite{LY-2Dobstacle}. Apparently, the analysis is more sophisticated and the computation is more intensive for the three-dimensional problem. This work adds a significant contribution to designing efficient computational methods for solving the elastic wave scattering problems.

The paper is organized as follows. In Section \ref{Section_pf}, the elastic wave equation is introduced; the boundary value problem is formulated by using the TBC; the corresponding weak formulation is discussed. In Section \ref{Section_dp}, the discrete problem is considered by using the finite element approximation with the truncated DtN operator. Section \ref{Section_pea} is devoted to the a posteriori error analysis and serves as the basis of the adaptive algorithm. In Section \ref{Section_ne}, we discuss the numerical implementation of the adaptive algorithm, the construction of the stiffness matrix, and an efficient solver for the linear system; two numerical examples are presented to illustrate the performance of the proposed method. The paper concludes with some general remarks in Section \ref{Section_c}. 

\section{Problem formulation}\label{Section_pf}

Let $D\subset\mathbb R^3$ be an elastically rigid obstacle with Lipschitz continuous boundary $\partial D$. Denote by $\nu$ the unit outward normal vector on $\partial D$. The exterior domain $\mathbb{R}^3\setminus \overline{D}$ is assumed to be filled with a homogeneous and isotropic elastic medium with a unit mass density. Let $B_R=\left\{\boldsymbol{x}\in\mathbb{R}^3: 
|\boldsymbol{x}|<R\right\}$ and $B_{R'}=\left\{\boldsymbol{x}\in\mathbb{R}^3: |\boldsymbol{x}|<R'\right\}$ be balls with radii $R$ and $R'$, where $0<R'<R$. Denote by $\Gamma_R$ and $\Gamma_{R'}$ the surfaces of $B_R$ and $B_{R'}$, respectively. Let $\Omega=B_{R}\setminus\overline{D}$ be the bounded domain enclosed by the surfaces $\partial D$ and $\Gamma_R$. 

Let the obstacle be illuminated by an incident wave $\boldsymbol{u}^{\rm inc}$, which can be a point source or a plane wave. Due to the interaction between the incident wave and the obstacle, the displacement of the scattered field $\boldsymbol{u}$ satisfies the elastic wave equation
\begin{equation}\label{section2_problem11}
\mu\Delta\boldsymbol{u}+(\lambda+\mu)\nabla\nabla\cdot\boldsymbol{u}+\omega^2\boldsymbol{u}=0 \quad {\rm in} ~ 
\mathbb{R}^3\setminus\overline{D},
\end{equation}
where $\omega>0$ is the angular frequency and $\lambda, \mu$ are the Lam\'e constants satisfying $\mu>0, \lambda+\mu>0$. Since the obstacle is assumed to be elastically rigid, the displacement of the total field vanishes on the surface of the obstacle, i.e., we have
\begin{equation*}
\boldsymbol{u}=\boldsymbol{g}=-\boldsymbol{u}^{\rm inc}\quad {\rm on} ~ \partial D.
\end{equation*}

Introduce the Helmholtz decomposition 
\begin{equation}\label{section2_problem2}
\boldsymbol{u}=\nabla\phi+\nabla\times\boldsymbol{\psi}, \qquad \nabla\cdot\boldsymbol{\psi}=0,
\end{equation}
where $\phi$ and $\boldsymbol{\psi}$ are called the potential functions. Substituting \eqref{section2_problem2} into \eqref{section2_problem11}, we may verify that the scalar potential $\phi$ satisfies the Helmholtz equation and the vector potential $\boldsymbol\psi$ satisfies the Maxwell equation in $\mathbb R^3\setminus\overline D$, i.e., 
\begin{equation*}
\Delta \phi+\kappa_p^2 \phi=0, \quad \nabla\times(\nabla\times\boldsymbol{\psi})-\kappa_s^2\boldsymbol{\psi}=0,
\end{equation*}
where 
\[
\kappa_p=\omega/(\lambda+2\mu)^{1/2}, \quad \kappa_s=\omega/\mu^{1/2}
\]
are known as the the compressional wavenumber and the shear wavenumber, respectively. In addition, the potential functions $\phi$ and $\boldsymbol\psi$ are required to satisfy the Sommerfeld radiation condition and the Silver--M\"{u}ller radiation condition, respectively:  
\begin{equation*}
\lim\limits_{\rho\rightarrow\infty} \rho\left(\partial_{\rho}\phi-{\rm i}\kappa_p\phi\right)=0,\quad
\lim\limits_{\rho\rightarrow\infty} \left(\left(\nabla\times\boldsymbol{\psi}\right)\times\boldsymbol{x}-
{\rm i}\kappa_s \rho\boldsymbol{\psi}\right)=0,\quad \rho=|\boldsymbol{x}|,
\end{equation*}
which is known as the Kupradze--Sommerfeld radiation condition for the elastic wave equation. 

The obstacle scattering problem is defined in the unbounded domain $\mathbb{R}^3\setminus\overline{D}$. It needs to be reduced equivalently into the bounded domain $\Omega$. Next we introduce a transparent boundary condition on $\Gamma_R$. Due to the page limit, the details are given as supplementary materials. Define a boundary operator for the  displacement of the scattered wave 
\[
\mathscr{D}\boldsymbol{u}:=\mu\partial_{\rho} \boldsymbol{u}+(\lambda+\mu)(\nabla\cdot\boldsymbol{u})\boldsymbol e_{\rho}
\quad {\rm on} ~ \Gamma_R.
\]
In the spherical coordinates, the scattered field admits the following expansion in $\mathbb R^3\setminus\overline{B_R}$:
\[
\boldsymbol{u}(R, \theta, \varphi)=\sum\limits_{n=0}^{\infty}\sum\limits_{m=-n}^n
u_{1n}^m\boldsymbol U_n^m(\theta, \varphi)+u_{2n}^m\boldsymbol V_n^m(\theta, \varphi)+u_{3n}^m X_n^m(\theta, \varphi)\boldsymbol e_{\rho}.
 \]
The transparent boundary condition is 
\begin{equation}\label{section2_TBC}
\mathscr{D}\boldsymbol{u}=\mathscr{T}\boldsymbol{u}:=\sum\limits_{n=0}^{\infty}\sum\limits_{m=-n}^n
b_{1n}^m \boldsymbol U_n^m+b_{2n}^m \boldsymbol V_n^m+b_{3n}^m X_n^m \boldsymbol e_{\rho}\quad\text{on} ~ \Gamma_R,
\end{equation}
where the Fourier coefficients $\boldsymbol b_n^m=(b_{1n}^m, b_{2n}^m, b_{3n}^m)^\top$ and $\boldsymbol u_n^m=(u_{1n}^m, u_{2n}^m, u_{3n}^m)^\top$ are connected by $\boldsymbol b_n^m=M_n \boldsymbol u_n^m$ with the $3\times 3$ matrix $M_n$ being given in Appendix \ref{appendixDtN}, and $\mathscr{T}$ is called the DtN operator.

Based on the transparent boundary condition \eqref{section2_TBC}, the scattering problem can be reformulated as the variational problem: find $\boldsymbol{u}\in \boldsymbol{H}^1(\Omega)$ with $\boldsymbol{u}=\boldsymbol{g}$
on $\partial D$ such that
\begin{equation}\label{section2_variational1}
b(\boldsymbol{u}, \boldsymbol{v})=0 \quad \forall\, \boldsymbol{v}\in \boldsymbol{H}_{\partial D}^1(\Omega),
\end{equation}
where the sesquilinear form $b: \boldsymbol{H}^1(\Omega)\times \boldsymbol{H}^1(\Omega)\rightarrow \mathbb{C}$ is defined as
\begin{align*}
b(\boldsymbol{u}, \boldsymbol{v})=\mu\int_{\Omega}\nabla\boldsymbol{u}:\nabla\overline{\boldsymbol{v}}{\rm d}\boldsymbol{x}
+(\lambda+\mu)\int_{\Omega}\left(\nabla\cdot\boldsymbol{u}\right)\left(\nabla\cdot\overline{\boldsymbol{v}}\right){\rm d}\boldsymbol{x} \notag\\
-\omega^2\int_{\Omega} \boldsymbol{u}\cdot\overline{\boldsymbol{v}}{\rm d}\boldsymbol{x}
-\int_{\Gamma_R} \mathscr{T}\boldsymbol{u}\cdot\overline{\boldsymbol{v}}{\rm d}s.
\end{align*}
Here $A:B={\rm tr}\left(AB^\top\right)$ is the Frobenius inner product of square matrices $A$ and $B$. 

It is shown in \cite{LY-ipi} that the variational problem admits a unique weak solution $\boldsymbol u\in \boldsymbol H^1(\Omega)$, which satisfies the estimate
\begin{equation}\label{section2_stability}
\|\boldsymbol{u}\|_{\boldsymbol{H}^1(\Omega)}\lesssim \|\boldsymbol{g}\|_{\boldsymbol{H}^{1/2}(\partial D)}
\lesssim \|\boldsymbol{u}^{\rm inc}\|_{\boldsymbol{H}^1(\Omega)}.
\end{equation}
Hereafter, the notation $a\lesssim b$ stands for $a \leq Cb$, where $C > 0$ is a generic constant whose value is not required and may change step by step in the proofs.

Since the variational problem is well-posed, it follows from the general theory in \cite{BA-AP-1973} that there exists a constant $\gamma>0$ such that the following inf-sup condition holds: 
\[
\sup\limits_{0\neq\boldsymbol{v}\in \boldsymbol{H}^1(\Omega)}\frac{|b(\boldsymbol{u}, \boldsymbol{v})|}{\|\boldsymbol{v}\|_{H^1(\Omega)}}\geq \gamma \|\boldsymbol{u}\|_{H^1(\Omega)}\quad \forall\, \boldsymbol{u}\in \boldsymbol{H}^1(\Omega).
\]

\section{The discrete problem}\label{Section_dp}

The DtN operator $\mathscr{T}$ in \eqref{section2_TBC} is given as an infinite series, which needs to be truncated into a sum of 
finitely many terms in computation. Given a sufficiently large $N$, define the truncated DtN operator
\begin{equation}\label{section3_truncDtN}
\mathscr{T}_N\boldsymbol{u}=\sum\limits_{n=0}^{N} \sum\limits_{m=-n}^n b_{1n}^m \boldsymbol U_n^m+b_{2n}^m \boldsymbol V_n^m+b_{3n}^m X_n^m \boldsymbol e_{\rho}. 
\end{equation}
The truncated variational problem is to find $\boldsymbol{u}_N\in \boldsymbol{H}^1(\Omega)$ with $\boldsymbol{u}_N=\boldsymbol{g}$ on $\partial D$ such that
\begin{equation}\label{variationalinter}
b_N(\boldsymbol{u}_N, \boldsymbol{v})=0 \quad \forall\, \boldsymbol{v}\in \boldsymbol{H}_{\partial D}^1(\Omega),
\end{equation}
where the sesquilinear form $b_N: \boldsymbol{H}^1(\Omega)\times \boldsymbol{H}^1(\Omega)\rightarrow \mathbb{C}$ is given by 
\begin{eqnarray*}
b_N(\boldsymbol{u}, \boldsymbol{v})=\mu\int_{\Omega}\nabla\boldsymbol{u}:\nabla\overline{\boldsymbol{v}}{\rm d}\boldsymbol{x}
+(\lambda+\mu)\int_{\Omega}\left(\nabla\cdot\boldsymbol{u}\right)\left(\nabla\cdot\overline{\boldsymbol{v}}\right){\rm d}\boldsymbol{x} \notag\\
-\omega^2\int_{\Omega} \boldsymbol{u}\cdot\overline{\boldsymbol{v}}{\rm d}\boldsymbol{x}-\int_{\Gamma_R} 
\mathscr{T}_N\boldsymbol{u}\cdot\overline{\boldsymbol{v}}{\rm d}s.
\end{eqnarray*}

Let $\mathcal{M}_h$ be a regular tetrahedral mesh of $\Omega$, where $h$ denotes the maximum diameter of all the elements in $\mathcal{M}_h$. For simplicity, we assume that the surfaces $\partial D$ and $\Gamma_R$ are polyhedral and ignore the approximation error on the surfaces, which allows us to focus on deducing the a posteriori error estimate. Thus any
face $e\in \mathcal{M}_h$ is a subset of $\partial \Omega$ if it has three
boundary vertices. 

Let $\boldsymbol{V}_h\subset \boldsymbol{H}^1(\Omega)$ be a conforming finite element space, i.e.,
\[
\boldsymbol{V}_h:=\left\{\boldsymbol{v}\in C(\overline{\Omega})^3: \boldsymbol{v}|_T\in P_m(T)^3 ~ \forall\, T\in \mathcal{M}_h\right\},
\]
where $m$ is a positive integer and $P_m(T)$ denotes the set of all polynomials of degree no more than $m$. The finite element approximation to the variational problem \eqref{variationalinter} is to find $\boldsymbol{u}_N^h\in\boldsymbol{V}_h$ with $\boldsymbol{u}_N^h=\boldsymbol{g}^h$ on $\partial D$ such that
\begin{equation}\label{section3_dis1}
	b_N(\boldsymbol{u}_N^h, \boldsymbol{v}_N^h)=0\quad \forall\, \boldsymbol{v}^h\in\boldsymbol{V}_{h, \partial D},
\end{equation}
where $\boldsymbol{g}^h$ is the finite element approximation of $\boldsymbol{g}$ and $\boldsymbol{V}_{h, \partial D}=\left\{\boldsymbol{v}\in\boldsymbol{V}_h: \boldsymbol{v}=0\,{\rm on}\,\partial D\right\}$.

Following the idea in \cite{HN-2011-jcam} and the discussion of \cite{LY-ipi}, we may show that for sufficiently large 
$N$ the variational problem \eqref{variationalinter} is well-posed.  Meanwhile, for sufficiently small $h$, the discrete inf-sup condition of the sesquilinear form $b_N$ may also be established by following the approach in \cite{S-mc74}. Based on the general theory in  \cite{BA-AP-1973}, the truncated variational problem \eqref{section3_dis1} can be shown to  have a unique solution $u_N^h\in\boldsymbol{V}_h.$ The details are omitted since our focus is the a posteriori error estimate and the convergence analysis for the truncated DtN operator.  

\section{The a posteriori error analysis}\label{Section_pea}

For any tetrahedral element $T\in \mathcal{M}_h$, denoted by $h_T$ its diameter. Let $\mathcal{B}_F$ denote the set of all the face of $T$ and $h_F$ be the size of the face $F$. For any interior face $F$ which is the common face of tetrahedral element $T_1, T_2\in\mathcal{M}_h$, define the jump residual across $F$ as
\[
J_F=-\left[\mu\nabla\boldsymbol{u}_N^h\cdot\nu_1+(\lambda+\mu)(\nabla\cdot\boldsymbol{u}_N^h) \nu_1+\mu\nabla\boldsymbol{u}_N^h\cdot\nu_2+(\lambda+\mu)(\nabla\cdot\boldsymbol{u}_N^h) \nu_2\right],
\]
where $\nu_j$ is the unit outward normal vector on the boundary of $T_j, j=1,2.$ For any boundary edge $F\subset \Gamma_R$, the jump residual is
\[
J_F=2\left(\mathscr{T}_N\boldsymbol{u}_N^h-\mu \nabla\boldsymbol{u}_N^h\cdot\boldsymbol e_{\rho}-(\lambda+\mu)(\nabla\cdot\boldsymbol{u}_N^h)\boldsymbol e_{\rho}\right),
\]
For any tetrahedral element $T\in \mathcal{M}_h$, define the local error estimator as
\[
\eta_{T}=h_T\|\mathscr{R} \boldsymbol{u}_N^h\|_{\boldsymbol{L}^2(T)}+\left(\frac{1}{2}\sum\limits_{F\in\partial T}h_F \|J_F\|^2_{\boldsymbol{L^2}(F)}\right)^{1/2},
\]
where $\mathscr{R}$ is the residual operator given by 
\[
\mathscr{R}\boldsymbol{u}_N^h=\mu\Delta\boldsymbol{u}_N^h+(\lambda+\mu)\nabla(\nabla\cdot\boldsymbol{u}_N^h)
+\omega^2\boldsymbol{u}_N^h.
\]

Introduce the following weighted norm $\vvvert \cdot\vvvert_{\boldsymbol{H}^1(\Omega)}$:
\begin{equation}\label{section4_weight}
\vvvert \boldsymbol{u} \vvvert_{\boldsymbol{H}^1(\Omega)}^2=\mu\int_{\Omega}|\nabla\boldsymbol{u}|^2{\rm d}\boldsymbol{x}
+(\lambda+\mu)\int_{\Omega}|\nabla\cdot\boldsymbol{u}|^2{\rm d}\boldsymbol{x}+\omega^2 \int_{\Omega}|\boldsymbol{u}|^2{\rm d}\boldsymbol{x}.
\end{equation}
It is easy to check that for any $\boldsymbol{u}\in \boldsymbol{H}^1(\Omega)$ we have 
\begin{equation*}
	\min(\mu, \omega^2)\|\boldsymbol{u}\|_{\boldsymbol{H}^1(\Omega)}^2\leq
	\vvvert \boldsymbol{u} \vvvert^2_{\boldsymbol{H}^1(\Omega)}
	\leq \max(2\lambda+3\mu, \omega^2)\|\boldsymbol{u}\|^2_{\boldsymbol{H}^1(\Omega)},
\end{equation*}
which implies that the norms $\|\cdot\|_{\boldsymbol H^1(\Omega)}$ and  $\vvvert \cdot\vvvert_{\boldsymbol{H}^1(\Omega)}$ are equivalent.

Now, let us state the main result of this paper.

\begin{theorem}\label{Mainthm}
Let $\boldsymbol{u}$ and $\boldsymbol{u}_N^h$ be the solutions of the variational problems \eqref{section2_variational1} and \eqref{section3_dis1}, respectively. Let $\boldsymbol{\xi}=\boldsymbol{u}-\boldsymbol{u}_N^h$. Then for sufficiently large $N$,
the following a posteriori error estimate holds:
\[
\vvvert \boldsymbol{\xi} \vvvert_{\boldsymbol{H}^1(\Omega)}\lesssim \left(\sum\limits_{T\in \mathcal{M}_h} \eta_T^2\right)^{1/2}+\|\boldsymbol{g}-\boldsymbol{g}^h\|_{\boldsymbol{H}^{1/2}(\partial D)}+
N \left(\frac{R'}{R}\right)^{N}\|\boldsymbol u^{\rm inc}\|_{\boldsymbol{H}^1(\Omega)}.
\]
\end{theorem}

It can be seen from the theorem that the a posteriori error estimate consists of three parts: the first two parts arise from the finite element discretization error; the third part accounts for the truncation error of the DtN operator. Apparently, the DtN truncation error decreases exponentially with respect to $N$ since
$R'<R$.

Using \eqref{section4_weight} and the integration by parts, we obtain 
\begin{align}\label{section4_key}
\vvvert\boldsymbol{\xi}\vvvert^2_{\boldsymbol{H}^{1}(\Omega)}&= \mu\int_{\Omega} \nabla\boldsymbol{\xi}:\nabla\overline{\boldsymbol{\xi}}{\rm d}\boldsymbol{x}+(\lambda+\mu)\int_{\Omega}(\nabla\cdot\boldsymbol{\xi}
)(\nabla\cdot\overline{\boldsymbol{\xi}}){\rm d}\boldsymbol{x}+\omega^2 \int_{\Omega}
\boldsymbol{\xi}\cdot\overline{\boldsymbol{\xi}}{\rm d}\boldsymbol{x}\notag\\
&= \Re b(\boldsymbol{\xi}, \boldsymbol{\xi})+2\omega^2 \int_{\Omega}\boldsymbol{\xi}\cdot\overline{\boldsymbol{\xi}}{\rm
d}\boldsymbol{x} +\Re\int_{\Gamma_R}\mathscr{T}\boldsymbol{\xi}\cdot \overline{\boldsymbol{\xi}}{\rm d}s\notag\\
&= \Re b(\boldsymbol{\xi}, \boldsymbol{\xi})+\Re\int_{\Gamma_R}\left(\mathscr{T} -\mathscr{T}_N\right)\boldsymbol{\xi}\cdot\overline{\boldsymbol{\xi}}{\rm d}s +2\omega^2 \int_{\Omega}\boldsymbol{\xi}\cdot\overline{\boldsymbol{\xi}}{\rm
d}\boldsymbol{x} +\Re\int_{\Gamma_R}\mathscr{T}_N \boldsymbol{\xi}\cdot\overline{\boldsymbol{\xi}}{\rm d}s.
\end{align}
To prove Theorem \ref{Mainthm}, it suffices to estimate the four terms given on the right-hand side of \eqref{section4_key}. 

The following lemma concerns the trace theorem. The proof is standard and omitted for brevity.  

\begin{lemma}\label{Poincare}
For any $u\in H^{1}(\Omega)$, the following estimates hold: 
\[
\|u\|_{H^{1/2}(\Gamma_R)}\lesssim \|u\|_{H^{1}(\Omega)},\quad \|u\|_{H^{1/2}(\Gamma_{R'})}\lesssim
\|u\|_{H^{1}(\Omega)}.
\]
\end{lemma}

\begin{lemma}\label{Firsttwo1}
Let $\boldsymbol{u}\in\boldsymbol{H}^1(\Omega)$ be the solution of the variational problem \eqref{section2_variational1}. For any $\boldsymbol{v}\in \boldsymbol{H}^1(\Omega)$, the following estimate holds: 
\[
\left|\int_{\Gamma_R} \left(\mathscr{T}-\mathscr{T}_N\right)\boldsymbol{u}\cdot\boldsymbol{v}{\rm d}s\right|\leq C N\left(\frac{R'}{R}\right)^{N}\|\boldsymbol u^{\rm inc}\|_{\boldsymbol{H}^1(\Omega)}\|\boldsymbol{v}\|_{\boldsymbol{H}^1(\Omega)} , 
\]
where $C$ is a positive constant independent of $N$.
\end{lemma}

\begin{proof}
Let $\phi$ and $\boldsymbol{\psi}$ be the potentials of the Helmholtz decomposition for the solution $\boldsymbol{u}$.  It can be verified from \eqref{AppDtN_phi}--\eqref{AppDtN_psi} that
\begin{eqnarray}\label{section4_RR1}
\phi_n^m(R)=z_n^p\phi_n^m(R'),\quad \psi_{2n}^m(R)=z_n^s\psi_{2n}^m(R'),\quad \psi_{3n}^m(R)=\left(\frac{R'}{R}\right)z_n^s \psi_{3n}^m(R'),
\end{eqnarray}
where
\[
 z_n^p=\frac{h_n^{(1)}(\kappa_p R)}{h_n^{(1)}(\kappa_p R')},\quad  z_n^s=\frac{h_n^{(1)}(\kappa_s R)}{h_n^{(1)}(\kappa_s R')}.
\]
Substituting \eqref{section4_RR1} into \eqref{DtN_Kn} and using \eqref{DtN_Kninverse} with $R$ being replaced with $R'$, we obtain
\begin{equation}\label{section4_uRR1}
 \boldsymbol u^m_n(R)=Q_n \boldsymbol u^m_n(R'),
\end{equation}
where the entries of the matrix $Q_n$ are 
\begin{align*}
Q_{n, 11}&=\frac{R'}{R\Lambda_n(R')}\left[-n(n+1)z_n^p +z_n^{(1)}(\kappa_p R')(1+z_n^{(1)}(\kappa_s R))z_n^s\right],\\
Q_{n, 13}&= \frac{R'}{R\Lambda_n(R')}\sqrt{n(n+1)}\left[ (1+z_n^{(1)}(\kappa_s R'))z_n^p-(1+z_n^{(1)}(\kappa_s R))z_n^s\right],\\
Q_{n, 31}&= \frac{R'}{R\Lambda_n(R')}\sqrt{n(n+1)}\left[- z_n^{(1)}(\kappa_p R) z_n^p +z_n^{(1)}(\kappa_p R')z_n^s \right],\\
Q_{n, 33}&=\frac{R'}{R\Lambda_n(R')}\left[(1+z_n^{(1)}(\kappa_s R'))z_n^{(1)}(\kappa_p R)z_n^p-n(n+1)z_n^s \right],
\end{align*}
and
\begin{eqnarray*}
Q_{n, 22}=z_n^s, \quad Q_{n, 12} =Q_{n, 21}=Q_{n, 23}=Q_{n, 32}=0.
\end{eqnarray*}

We use the element $Q_{n, 11}$ as an example to show the estimate of the matrix $Q_n$ since all the other elements can be similarly estimated. A simple calculation yields 
\begin{align*}
Q_{n, 11}&=\frac{R'}{R\Lambda_n(R')}n(n+1) (z_n^s -z_n^p) \\
&\quad+\frac{R'}{R\Lambda_n(R')}\left[z_n^{(1)}(\kappa_p R')(1+z_n^{(1)}(\kappa_s R))
-n(n+1)\right]z_n^s.
\end{align*}
It follows from Lemma \ref{Bessel5} and \eqref{Lambdan}--\eqref{Thm2_bessel} that
\begin{eqnarray*}
\left|Q_{n, 11}\right|\lesssim n\left(\frac{R'}{R}\right)^n.
\end{eqnarray*}
Similarly, we may show that all the other entries of $K_n$ satisfy 
\[
\left|Q_{n, ij}\right|\lesssim n\left(\frac{R'}{R}\right)^n,\quad i, j=1,2,3.
\]

Substituting the estimate of $Q_n$ into \eqref{section4_uRR1}, we have
\[
\left|\boldsymbol{u}^m_{n}(R)\right|\lesssim n\left(\frac{R'}{R}\right)^n \left|\boldsymbol{u}^m_{n}(R')\right|.
\]
Combining the above estimate with \eqref{Mnestimate} and using Theorem \ref{Poincare} and \eqref{section2_stability}, we have 
\begin{align*}
\left|\int_{\Gamma_R} \left(\mathscr{T}-\mathscr{T}_N\right)\boldsymbol{u}\cdot\boldsymbol{v}\,{\rm d}s\right| &= \left|
\sum\limits_{n>N}\sum\limits_{m=-n}^n M_n \boldsymbol{u}^m_n(R) \cdot\overline{\boldsymbol{v}}^m_n(R)\right|\\
&=\left|\sum\limits_{n>N}\sum\limits_{m=-n}^n M_n Q_n\boldsymbol{u}^m_n(R') \cdot\overline{\boldsymbol{v}}^m_n(R)\right|\\
&\lesssim \left|\sum\limits_{n>N}\sum\limits_{m=-n}^n n^2 \left(\frac{R'}{R}\right)^n\boldsymbol{u}^m_n(R') \cdot\overline{\boldsymbol{v}}^m_n(R)\right|\\
&\lesssim  \max\limits_{n>N}\left(n \left(\frac{R'}{R}\right)^n\right)\|\boldsymbol u^{\rm inc}\|_{\boldsymbol{H}^1(\Omega)}\|\boldsymbol{v}\|_{\boldsymbol{H}^1(\Omega)}.
\end{align*}
The proof is completed by noting that $n \left(R'/R\right)^n$ decreases for sufficiently large $n$.
\end{proof}

Based on Lemma \ref{Firsttwo1}, the estimate of the first two terms in \eqref{section4_key} is given in the following lemma. 
The proof follows directly from the discussion in \cite{LY-2Dobstacle}, where the basic idea is to use the integration
by parts and the interpolation theory. The details are omitted. 

\begin{lemma}\label{Firsttwo2}
Let $\boldsymbol{v}$ be any function in $\boldsymbol{H}^{1}_{\partial D}(\Omega)$, the following estimate holds: 
\begin{equation*}
\left|b(\boldsymbol{\xi}, \boldsymbol{v})+\int_{\Gamma_R}\left(\mathscr{T}-\mathscr{T}_N\right)\boldsymbol{\xi}\cdot 
\overline{\boldsymbol{v}}{\rm d}s\right|\lesssim \left( \left(\sum\limits_{T\in\mathcal M_h}\eta_{T}^2\right)^{1/2}
+N\left(\frac{R'}{R}\right)^{N} \|\boldsymbol u^{\rm inc}\|_{\boldsymbol{H}^{1}(\Omega)}\right)\|\boldsymbol{v}\|_{\boldsymbol{H}^{1}(\Omega)},
\end{equation*}
which gives by taking $\boldsymbol{v}=\boldsymbol{\xi}$ that 
\begin{align*}
\left|b(\boldsymbol{\xi}, \boldsymbol{\xi})+\int_{\Gamma_R}\left(\mathscr{T}-\mathscr{T}_N\right)\boldsymbol{\xi}\cdot\overline{\boldsymbol{\xi}}\,{\rm d}s\right|&\lesssim \Bigg(\left(\sum\limits_{T\in M_h} \eta_T^2\right)^{1/2}
+N\left(\frac{R'}{R}\right)^N \|\boldsymbol u^{\rm inc}\|_{\boldsymbol{H}^1(\Omega)} \notag\\
&\quad +\|\boldsymbol{g}-\boldsymbol{g}^h\|_{\boldsymbol{H}^{1/2}(\partial D)}\Bigg)
\|\boldsymbol{\xi}\|_{\boldsymbol{H}^1(\Omega)}. 
\end{align*}
\end{lemma}

It is proved in \cite{LY-ipi} that the matrix $\hat{M}_n=-\frac{1}{2}\left(M_n+M_n^*\right)$ is positive definite for sufficiently large $n$, where the star denotes the complex transpose. The following result can be obtained easily by following the proof of \cite[Lemma 4.6]{LY-2Dobstacle}. 

\begin{lemma}\label{Lastterm}
For any $\delta>0$, there exists a positive constant $C(\delta)$ independent of $N$ such that
\[
\int_{\Gamma_R} \mathscr{T}_N \boldsymbol{\xi}\cdot\overline{\boldsymbol{\xi}}{\rm d}s\leq C(\delta) \|\boldsymbol{\xi}\|^2_{\boldsymbol{L}^2(B_{R}\setminus B_{R'})}+\left(\frac{R}{R'}\right)\delta \|\boldsymbol{\xi}\|^2_{\boldsymbol{H}^1(B_R\setminus B_{R'})}.
\]
\end{lemma}  

To estimate the third term of \eqref{section4_key}, we introduce the dual problem
\begin{equation}\label{section4_dual1}
b(\boldsymbol{v}, \boldsymbol{p})=\int_{\Omega} \boldsymbol{v}\cdot\overline{\boldsymbol{\xi}}{\rm d}\boldsymbol x\quad
\forall\, \boldsymbol{v}\in \boldsymbol{H}_{\partial D}^1(\Omega).
\end{equation}
It is easy to verify that $\boldsymbol{p}$ satisfies the boundary value problem
\begin{equation}\label{section4_dual1problem}
\begin{cases}
\mu\Delta \boldsymbol{p}+(\lambda+\mu)\nabla\nabla\cdot\boldsymbol{p}+\omega^2\boldsymbol{p}=-\boldsymbol{\xi} \quad & {\rm in} ~ \Omega,\\
\boldsymbol{p}=0 \quad & {\rm on} ~ \partial D,\\
\mathscr{D}\boldsymbol{p}=\mathscr{T}^*\boldsymbol{p}\quad &{\rm on} ~ \Gamma_R,
\end{cases}
\end{equation}
where $\mathscr{T}^*$ is the adjoint operator to the DtN operator $\mathscr{T}$. 

Letting $\boldsymbol{v}=\boldsymbol{\xi}$ in \eqref{section4_dual1} leads to 
\begin{equation}\label{section4_variational}
\|\boldsymbol{\xi}\|^2_{\boldsymbol{L}^2(\Omega)}=b(\boldsymbol{\xi}, \boldsymbol{p})+\int_{\Gamma_R} \left(\mathscr{T}-\mathscr{T}_N\right)\boldsymbol{\xi}\cdot\overline{\boldsymbol{p}}{\rm d}s-\int_{\Gamma_R} \left(\mathscr{T}-\mathscr{T}_N\right)\boldsymbol{\xi}\cdot\overline{\boldsymbol{p}}{\rm d}s.
\end{equation}
The first two terms of \eqref{section4_variational} can be estimated in the same way as that of Lemma \ref{Firsttwo2}. Thus it suffices to estimate the third term of  \eqref{section4_variational}. Next we deduce the solution of the dual problem \eqref{section4_dual1} which is crucial for the estimate.

Consider the system 
\begin{equation}\label{section4_Helmholtzxi}
\begin{cases}
\nabla\zeta+\nabla\times\boldsymbol{Z} = \boldsymbol{\xi}, \quad \nabla\cdot\boldsymbol{Z}=0 \quad & \text{in} ~ B_{R}\setminus \overline{B_{R'}},\\
\zeta(R)=0,\quad  \boldsymbol{Z}(R)=0\quad & \text{on} ~ \Gamma_R.\\
\end{cases}
\end{equation}
By a straightforward calculation, the Fourier coefficients for the solution of \eqref{section4_Helmholtzxi} are given by 
\begin{align}
\zeta_n^m(\rho) &= \frac{1}{2n+1}\int_{\rho}^R\left[-n c_2-(n+1) c_4\right]\xi_{3n}^m(\tau)-\sqrt{n(n+1)}(c_2-c_4)\xi_{1n}^m(\tau){\rm d}\tau, \label{section4_zeta}\\
Z_{1n}^{m}(\rho) &= \frac{1}{2n+1}\int_{\rho}^R \left[-n c_1-(n+1) c_3\right]\xi_{2n}^m(\tau){\rm d}\tau,  \label{section4_Z1nm}\\
Z_{2n}^{m}(\rho) &= \frac{1}{2n+1}\int_{\rho}^R\sqrt{n(n+1)}(c_2-c_4)\xi_{3n}^m(\tau)+\left[ (n+1)c_2+n c_4\right]\xi_{1n}^m(\tau){\rm d}\tau, \label{section4_Z2nm}\\
Z_{3n}^{m}(\rho) &= \frac{1}{2n+1}\int_{\rho}^R\sqrt{n(n+1)}(c_1-c_3) \xi_{2n}^m(\tau){\rm d}\tau, \label{section4_Z3nm}
\end{align}
where
\[
c_1=\left(\frac{\rho}{\tau}\right)^{-n-2}, \quad c_2=\left(\frac{\rho}{\tau}\right)^{-n-1},\quad c_3=
\left(\frac{\rho}{\tau}\right)^{n-1}, \quad c_4=\left(\frac{\rho}{\tau}\right)^{n}.
\]

Consider the following boundary value problem: 
\begin{equation}\label{section4_HelmholtzD}
\begin{cases}
\Delta g+\kappa_p^2 g=-\frac{1}{\lambda+2\mu}\zeta, \quad -\nabla\times(\nabla\times\boldsymbol{q})+\kappa_s^2\boldsymbol{q}=-\frac{1}{\mu}\boldsymbol{Z} \quad &{\rm in} ~ B_{R}\setminus B_{R'}, \\
\nabla\cdot \boldsymbol{q}=\nabla\cdot\boldsymbol{Z}=0 & {\rm in} ~ B_{R}\setminus B_{R'},\\
\partial_{\rho}g=\mathscr{T}_1^*g, \quad \left(\nabla\times\boldsymbol{q}\right)\times\boldsymbol e_{\rho}=-{\rm i}\kappa_s\mathscr{T}_2^* \boldsymbol{q}_{\Gamma_R} & {\rm on} ~ \Gamma_R,\\
g=g, \quad \boldsymbol{q}=\boldsymbol{q} &{\rm on} ~ \Gamma_{R'}, 
\end{cases}
\end{equation}
where $\mathscr{T}^*_1$ and $\mathscr{T}^*_2$ are the adjoint operators to the DtN operators $\mathscr{T}_1$  (cf. \eqref{DtN_T1}) and $\mathscr{T}_2$ (cf. \eqref{DtN_T2}), respectively, and $\boldsymbol q_{\Gamma_R}=-\boldsymbol e_\rho\times(\boldsymbol e_\rho\times \boldsymbol q)$ is the tangential component of $\boldsymbol  q$ on $\Gamma_R$.  

\begin{lemma}\label{step1}
If $(\zeta, \boldsymbol{Z})$ and $(g, \boldsymbol  q)$ are the solutions of the systems  \eqref{section4_Helmholtzxi} and
\eqref{section4_HelmholtzD}, respectively, then $\boldsymbol{p}=\nabla g+\nabla\times\boldsymbol{q}$ is the solution of the dual problem \eqref{section4_dual1problem} in $B_{R}\setminus\overline{B_{R'}}$.
\end{lemma}

\begin{proof}
Letting $\boldsymbol{p}=\nabla g+\nabla\times \boldsymbol{q}$ and substituting it into the elastic wave equation, we get 
\begin{align*}
& \mu\Delta\left(\nabla g+\nabla\times \boldsymbol{q}\right)+(\lambda+\mu)\nabla\nabla\cdot\left(\nabla g+\nabla\times \boldsymbol{q}\right)+\omega^2\left(\nabla g+\nabla\times \boldsymbol{q}\right)\\
&=\nabla\left((\lambda+2\mu)\Delta g+\omega^2 g\right)+\nabla\times\left(-\mu\nabla\times\nabla\times \boldsymbol{q}+\mu\nabla\nabla\cdot\boldsymbol{q}+\omega^2\boldsymbol{q}\right)\\
&=-\nabla\zeta-\nabla\times\boldsymbol{Z}=-\boldsymbol{\xi},
\end{align*}
which shows that $\boldsymbol{p}$ satisfies the elastic wave equation in \eqref{section4_dual1problem}. The rest of the proof is to show that $\boldsymbol{p}$ satisfies the boundary condition $\mathscr{D}\boldsymbol{p}=\mathscr{T}^*\boldsymbol{p}$ on $\Gamma_R$.

It follows from \eqref{gradspherical}--\eqref{curlspherical} that 
\begin{align}\label{section4_p}
\boldsymbol{p} &=\nabla g+\nabla\times\boldsymbol{q}\notag\\
&=  \sum\limits_{n\in\mathbb{N}}\sum\limits_{|m|\leq n}\left[\frac{\sqrt{n(n+1)}}{\rho}g_n^m(\rho)-\frac{1}{\rho} q_{2n}^m(\rho)-q_{2n}^{m'}(\rho)\right]\boldsymbol U_n^m \\
&\quad +\left[\frac{1}{\rho} q_{1n}^m(\rho)+q_{1n}^{m'}(\rho)-\frac{\sqrt{n(n+1)}}{\rho}q_{3n}^m(\rho) \right] 
\boldsymbol V_n^m +\left[g_{n}^{m'}(\rho)-\frac{\sqrt{n(n+1)}}{\rho}q_{2n}^m(\rho)\right]X_n^m\boldsymbol e_{\rho}. \notag
\end{align}
Let $v_n^m(\rho)=\rho q_{3n}^m(\rho)$. Since $\nabla\cdot\boldsymbol{q}=0$, we have from \eqref{divspherical} that
\begin{align*}
q_{1n}^m(\rho)&=\frac{1}{\sqrt{n(n+1)}}\frac{1}{\rho}\left(v_n^m(\rho)+\rho v_n^{m'}(\rho)\right),\\
q_{1n}^{m'}(\rho)&=\frac{1}{\sqrt{n(n+1)}}\left(-\frac{1}{\rho^2}v_n^m(\rho)+\frac{1}{\rho}v_n^{m'}(\rho)+ v_n^{m''}(\rho)\right). 
\end{align*}
It follows from Lemma \ref{step23} that
\begin{align}\label{q1q3}
\frac{1}{\rho} q_{1n}^m(\rho)+q_{1n}^{m'}(\rho)-\frac{\sqrt{n(n+1)}}{\rho}q_{3n}^m(\rho) 
&=\frac{1}{\sqrt{n(n+1)}}\left(v_n^{m''}(\rho)+\frac{2}{\rho}v_n^{m'}(\rho)-\frac{n(n+1)}{\rho^2} v_n^m(\rho)\right) \notag\\
&= -\frac{1}{\sqrt{n(n+1)}} \left(\beta_n^m(\rho)+\kappa_s^2 v_n^m(\rho)\right). 
\end{align}
Substituting \eqref{q1q3} into \eqref{section4_p} and taking derivative of $\boldsymbol{p}$, we obtain 
\begin{align}\label{section4_dpR}
\boldsymbol{p}'(\rho)& =  \sum\limits_{n\in\mathbb{N}}\sum\limits_{|m|\leq n}\left[\sqrt{n(n+1)}\frac{\partial}{\partial \rho}\left(\frac{1}{\rho}g_n^m(\rho)\right)-\frac{\partial}{\partial \rho}\left(\frac{1}{\rho} q_{2n}^m(\rho)\right)-q_{2n}^{m''}(\rho)\right]\boldsymbol U_n^m \\
&\quad -\frac{1}{\sqrt{n(n+1)}} \left[\beta_n^{m'}(\rho)+\kappa_s^2 v_n^{m'}(\rho)\right]\boldsymbol V_n^m+\left[g_{n}^{m''}(\rho)-\sqrt{n(n+1)}\frac{\partial}{\partial \rho}\left(\frac{1}{\rho}q_{2n}^m(\rho)\right)\right]X_n^m \boldsymbol e_{\rho}. \notag
\end{align}
A simple calculation yields 
\begin{equation}\label{section4_divp}
\nabla\cdot\boldsymbol{p}=\Delta g=-\frac{1}{\lambda+2\mu}\zeta-\kappa_p^2 g.
\end{equation}
Substituting \eqref{section4_dpR}--\eqref{section4_divp} into the boundary operator $\mathscr{D}$, we have 
\begin{align}\label{I123}
\mathscr{D}\boldsymbol{p} &=\mu\partial_{\rho}\boldsymbol{p}+(\lambda+\mu)(\nabla\cdot\boldsymbol{p})\boldsymbol e_{\rho}\notag\\
&=\sum\limits_{n\in\mathbb{N}}\sum\limits_{|m|\leq n}\left\{\left[\mu g_{n}^{m''}(\rho)-(\lambda+\mu)\kappa_p^2 g_n^m(\rho)\right]X_n^m\boldsymbol e_{\rho}+\mu\sqrt{n(n+1)}\frac{\partial}{\partial \rho}\left(\frac{1}{\rho}g_n^m(\rho)\right)\boldsymbol U_n^m \right\} \notag\\
&\quad +\mu\sum\limits_{n\in\mathbb{N}}\sum\limits_{|m|\leq n}\Bigg\{-\left[\frac{\partial}{\partial \rho}\left(\frac{1}{\rho} q_{2n}^m(\rho)\right)+q_{2n}^{m''}(\rho)\right]\boldsymbol U_n^m -\sqrt{n(n+1)}\frac{\partial}{\partial \rho}\left(\frac{1}{\rho}q_{2n}^m(\rho)\right)X_n^m \boldsymbol e_{\rho} \notag\\
&\quad -\frac{1}{\sqrt{n(n+1)}} \left[\beta_n^{m'}(\rho)+\kappa_s^2 v_n^{m'}(\rho)\right] \boldsymbol V_n^m \Bigg\}-\left(\frac{\lambda+\mu}{\lambda+2\mu}\right)\zeta\boldsymbol e_{\rho} \notag\\
&=:I_1+I_2+I_3,
\end{align}
where $I_1, I_2$ and $I_3$ accounts for the summation of $g_n^m$, $q_{2n}^m$ and $q_{3n}^m$, respectively.

It follows from \eqref{lapspherical} that
\begin{align*}
\Delta g+\kappa_p^2\, g=-\frac{1}{\lambda+2\mu}\zeta, 
\end{align*}
which gives 
\begin{align*}
\sum\limits_{n\in\mathbb{N}}\sum\limits_{|m|\leq n}\left[ g_n^{m''}(\rho)+\frac{2}{\rho}g_n^{m'}(\rho)-\frac{n(n+1)}{\rho^2} g_n^m(\rho)+\kappa_p^2 \,g_n^m\right]X_n^m+\frac{1}{\lambda+2\mu}\zeta=0. 
\end{align*}
Substituting the above equation into $I_1$ and replacing the second order derivative, we obtain 
\begin{align*}
I_1 &=  \sum\limits_{n\in\mathbb{N}}\sum\limits_{|m|\leq n}\left\{\left[\mu g_{n}^{m''}(\rho)-(\lambda+\mu)\kappa_p^2 g_n^m(\rho)\right]X_n^m \boldsymbol e_{\rho}+\mu\sqrt{n(n+1)}\frac{\partial}{\partial \rho}\left(\frac{1}{\rho}g_n^m(\rho)\right)\boldsymbol U_n^m \right\}\\
&= \sum\limits_{n\in\mathbb{N}}\sum\limits_{|m|\leq n}\Bigg\{\left[-\mu\frac{2}{\rho}g_n^{m'}(\rho)+\mu\frac{n(n+1)}{\rho^2} g_n^m(\rho)-\omega^2 \,g_n^m(\rho)\right]X_n^m\boldsymbol e_{\rho}\\
&\quad +\mu\sqrt{n(n+1)}\left[\frac{1}{\rho}g_n^{m'}(\rho)-\frac{1}{\rho^2} g_n^m(\rho)\right]\boldsymbol U_n^m\Bigg\}
\end{align*}
Replacing the first order derivative by the transparent boundary condition
\[
g_n^{m'}(R)=\frac{z_n^{(2)}(\kappa_p R)}{R} g_n^m(R), 
\]  
we have from a straightforward computation that
\begin{align*}
I_1(R) &= \frac{1}{R^2} \sum\limits_{n\in\mathbb{N}}\sum\limits_{|m|\leq n}\Bigg\{\left[-2\mu z_n^{(2)}(\kappa_p R)+n(n+1)\mu-\omega^2 R^2 \right] g_n^m(R) X_n^m\boldsymbol e_{\rho}\\
&\quad +\mu\sqrt{n(n+1)}\left[z_n^{(2)}(\kappa_p R)-1\right] g_n^m(R) \boldsymbol U_n^m\Bigg\}. 
\end{align*}

Using \eqref{section4_HelmholtzD} and \eqref{curlcurlspherical}, we get 
\[
-\frac{1}{\rho}\frac{\partial^2}{\partial \rho^2}\left(\rho q_{2n}^m(\rho)\right)+\frac{n(n+1)}{\rho^2} q_{2n}^m(\rho)-\kappa_s^2 q_{2n}^m(\rho)=\frac{1}{\mu}Z_{2n}^{m}(\rho), 
\]
which implies that 
\[
q_{2n}^{m''}(\rho)=-\frac{2}{\rho}q_{2n}^{m'}(\rho)+\frac{n(n+1)}{\rho^2} q_{2n}^m(\rho)-\kappa_s^2 q_{2n}^m(\rho)-\frac{1}{\mu}Z_{2n}^m(\rho). 
\]
Substituting the above equation into $I_2$ and replacing the second derivative lead to 
\begin{align*}
I_2 &= \mu\sum\limits_{n\in\mathbb{N}}\sum\limits_{|m|\leq n}\Bigg\{-\left[q_{2n}^{m''}(\rho)+\left(-\frac{1}{\rho^2} q_{2n}^m(\rho)+\frac{1}{\rho} q_{2n}^{m'}(\rho)\right)\right]\boldsymbol U_n^m \\
&\quad -\sqrt{n(n+1)}\left(-\frac{1}{\rho^2}q_{2n}^m(\rho)+\frac{1}{\rho}q_{2n}^{m'}(\rho)\right)X_n^m \boldsymbol e_{\rho} \Bigg\}\\
&= \mu\sum\limits_{n\in\mathbb{N}}\sum\limits_{|m|\leq n}\Bigg\{\left[\frac{2}{\rho}q_{2n}^{m'}(\rho)-\frac{n(n+1)}{\rho^2} q_{2n}^m(\rho)+\kappa_s^2 q_{2n}^m(\rho)+\frac{1}{\mu}Z_{2n}^m(\rho)+\frac{1}{\rho^2} q_{2n}^m(\rho)-\frac{1}{\rho} q_{2n}^{m'}(\rho)\right]\boldsymbol U_n^m \\
&\quad -\sqrt{n(n+1)}\left(-\frac{1}{\rho^2}q_{2n}^m(\rho)+\frac{1}{\rho}q_{2n}^{m'}(\rho)\right)X_n^m\boldsymbol e_{\rho} \Bigg\}\\
&=\mu\sum\limits_{n\in\mathbb{N}}\sum\limits_{|m|\leq n}\Bigg\{\left[\frac{1}{\rho}q_{2n}^{m'}(\rho)+\left( \kappa_s^2+\frac{1}{\rho^2}-\frac{n(n+1)}{\rho^2}\right) q_{2n}^m(\rho)\right]\boldsymbol U_n^m \\
&\quad -\sqrt{n(n+1)}\left[\frac{1}{\rho}q_{2n}^{m'}(\rho)-\frac{1}{\rho^2}q_{2n}^m(\rho)\right]X_n^m \boldsymbol e_{\rho}\Bigg\}.
\end{align*}
Note that the terms $Z_{1n}^m$ and $\zeta$ have been dropped since they vanish on $\Gamma_R$. Substituting the transparent boundary condition back to $I_2$ and eliminating the first order terms, we deduce
\begin{align*}
I_2(R) & =\frac{\mu}{R^2}\sum\limits_{n\in\mathbb{N}}\sum\limits_{|m|\leq n}\bigg\{\left[z_n^{(2)}(\kappa_s R) +\kappa_s^2 R^2
+1-n(n+1)\right] q_{2n}^m(R)\boldsymbol U_n^m \\
&\quad -\sqrt{n(n+1)}\left[z_n^{(2)}(\kappa_s R)-1\right]q_{2n}^m(R)X_n^m \boldsymbol e_{\rho} \bigg\}. 
\end{align*}

By \eqref{section4_Z3nm}, we have $\beta_n^{m'}(R)=0$. It follows from Lemma \ref{step23} that
\begin{align*}
I_3(R)&=-\mu\sum\limits_{n\in\mathbb{N}}\sum\limits_{|m|\leq n} \frac{1}{\sqrt{n(n+1)}} \left[\beta_n^{m'}(R)+\kappa_s^2 v_n^{m'}(R)\right]\boldsymbol V_n^m \\
&=-\sum\limits_{n\in\mathbb{N}}\sum\limits_{|m|\leq n} \frac{\mu\kappa_s^2 }{\sqrt{n(n+1)}} z_n^{(2)}(\kappa_s R) q_{3n}^{m}(R) \boldsymbol V_n^m. 
\end{align*}

Substituting $I_1, I_2$ and $I_3$ into \eqref{I123}, we get 
\begin{align}\label{Dp}
\mathscr{D}\boldsymbol{p} &= I_1 (R)+I_2(R)+I_3(R)\\
&=  \frac{1}{R^2} \sum\limits_{n\in\mathbb{N}}\sum\limits_{|m|\leq n}\Bigg\{\left[-2\mu \,z_n^{(2)}(\kappa_p R)+n(n+1)\mu-\omega^2 R^2 \right] g_n^m(R) X_n^m\boldsymbol e_{\rho} \notag\\
&\quad +\mu\sqrt{n(n+1)}\left[z_n^{(2)}(\kappa_p R)-1\right] g_n^m(R) \boldsymbol U_n^m\Bigg\} \notag\\
&\quad +\frac{\mu}{R^2}\sum\limits_{n\in\mathbb{N}}\sum\limits_{|m|\leq n}\Bigg\{\left[z_n^{(2)}(\kappa_s R) +\kappa_s^2 R^2
+1-n(n+1)\right] q_{2n}^m(R)\boldsymbol U_n^m \notag\\
&\quad -\sqrt{n(n+1)}\left[z_n^{(2)}(\kappa_s R)-1\right]q_{2n}^m(R)X_n^m\boldsymbol e_{\rho}
-\frac{z_n^{(2) }(\kappa_s R)\kappa_s^2 R^2}{\sqrt{n(n+1)}}
q_{3n}^m(R) \boldsymbol V_n^m \Bigg\} \notag
\end{align}
Note that \eqref{Dp} is the complex conjugate of (16) in \cite{LY-ipi}. Using the fact that the transparent boundary conditions of $g$ and $\boldsymbol{q}$ are the complex conjugate of the original transparent boundary conditions, we deduce $
\mathscr{D}\boldsymbol{p}=\mathscr{T}^*\boldsymbol{p}$ on $\Gamma_R$,  which completes the proof. 
\end{proof}

\begin{lemma}\label{3Ddual_estimatep}
Let
\[
\boldsymbol{p}(\rho, \theta, \varphi)=\sum\limits_{n=0}^{\infty}\sum\limits_{m=-n}^n p_{1n}^m(\rho)\boldsymbol U_n^m(\theta, \varphi)+p_{2n}^m(\rho)\boldsymbol V_n^m(\theta, \varphi)+p_{3n}^m(\rho) X_n^m(\theta, \varphi)\boldsymbol e_{\rho}
\]
be the solution of \eqref{section4_dual1problem} in $B_{R}\setminus\overline{ B_{R'}}$. Then the following estimate holds:
\[
|p_{jn}^m(R)|\lesssim n\left(\frac{R'}{R}\right)^{n}\sum\limits_{i=1}^3|p_{in}^m(R')|+\frac{1}{n}\|\boldsymbol{\xi}\|_{L^{\infty}([R', R])},\quad j=1,2,3.
\]
\end{lemma}

\begin{proof}
It follows from \eqref{section4_p}, \eqref{appendix_gR}, and \eqref{appendix_q2R} that 
\begin{align*}
p_{1n}^m(R) &= \frac{\sqrt{n(n+1)}}{R}g_n^m(R)-\frac{1}{R} q_{2n}^m(R)-\frac{1}{R}z_n^{(2)}(\kappa_s R) q_{2n}^m(R)\\
&=  \frac{\sqrt{n(n+1)}}{R}\left[S_n^p(R) g_n^m(R')+\frac{{\rm i}\kappa_p}{2}\int_{R'}^R t^2 S_n^p(R) W^p_n(R', t)\hat{\zeta}_n^m(t){\rm d}t\right]\\
&\quad -\frac{1}{R}\left[1+z_n^{(2)}(\kappa_s R)\right]\left[S_n^s(R) q_{2n}^m(R')+\frac{{\rm i}\kappa_s}{2}\int_{R'}^R t^2 S_n^s(R) W^s_n(R', t)\hat{Z}_{2n}^m(t){\rm d}t\right].
\end{align*}
Using \eqref{section4_p} and \eqref{appendix_q3R} yields 
\begin{align*}
p_{2n}^m(R) &= -\frac{1}{\sqrt{n(n+1)}}\kappa_s^2 R q_{3n}^m(R)\\
&= -\frac{1}{\sqrt{n(n+1)}}\kappa_s^2 R \left[\frac{R'}{R}S_n^s(R) q_{3n}^m(R')
+\frac{{\rm i}\kappa_s}{2R}\int_{R'}^R t^3 S_n^s(R) W^s_n(R', t)\hat{Z}_{3n}^m(t){\rm d}t\right]\\
&= -\frac{1}{\sqrt{n(n+1)}}\kappa_s^2  \left[R' S_n^s(R) q_{3n}^m(R')
+\frac{{\rm i}\kappa_s}{2}\int_{R'}^R t^3 S_n^s(R) W^s_n(R', t)\hat{Z}_{3n}^m(t){\rm d}t\right].
\end{align*}
We have from \eqref{section4_p}, \eqref{appendix_gR}, and \eqref{appendix_q2R} that
\begin{align*}
p_{3n}^m(R) &= \frac{1}{R}z_n^{(2)}(\kappa_p R) g_n^m(R)-\frac{\sqrt{n(n+1)}}{R}q_{2n}^m(R)\\
&=\frac{1}{R}z_n^{(2)}(\kappa_p R)\left[ S_n^p(R) g_n^m(R')
+\frac{{\rm i}\kappa_p}{2}\int_{R'}^R t^2 S_n^p(R) W^p_n(R', t)\hat{\zeta}_n^m(t){\rm d}t\right]\\
&\quad -\frac{\sqrt{n(n+1)}}{R}\left[S_n^s(R) q_{2n}^m(R')
+\frac{{\rm i}\kappa_s}{2}\int_{R'}^R t^2 S_n^s(R) W^s_n(R', t)\hat{Z}_{2n}^m(t){\rm d}t\right].
\end{align*}

Denote a diagonal matrix by
\[
M_{\rm diag}={\rm diag}\left(\frac{h_n^{(1)}(\kappa_p R)}{h_n^{(1)'}(\kappa_p R')},\quad \frac{h_n^{(1)}(\kappa_s R)}{h_n^{(1)'}(\kappa_s R')},\quad \frac{h_n^{(1)}(\kappa_s R)}{h_n^{(1)'}(\kappa_s R')}\right). 
\]
It is easy to check that
\begin{align}\label{pR}
\left(p_{1n}^m(R), p_{2n}^m(R), p_{3n}^m(R)\right)^\top&=\overline{K_n(R)} \overline{M_{\rm diag}}\left(g_n^m(R'), q_{2n}^m(R'), q_{3n}^m(R')\right)^\top \notag\\
&\quad +\overline{K_n(R)} \overline{M_{\rm diag}} \left(b_{1n}^m, b_{2n}^m, b_{3n}^m\right)^\top,
\end{align}
where
\begin{align*}
\boldsymbol{b}_n^m= \left(b_{1n}^m, b_{2n}^m, b_{3n}^m\right)^\top=\bigg(\frac{{\rm i}\kappa_p}{2}\int_{R'}^R t^2  W^p_n(R', t)\hat{\zeta}_n^m(t){\rm d}t,\, \frac{{\rm i}\kappa_s}{2}\int_{R'}^R t^2 W^s_n(R', t)\hat{Z}_{2n}^m(t){\rm d}t, \\
\, \frac{{\rm i}\kappa_s}{2RR'} \int_{R'}^R t^3 W^s_n(R', t)\hat{Z}_{3n}^m(t){\rm d}t \bigg)^\top.
\end{align*}

Next is to estimate $\boldsymbol{p}$ at $\rho=R'$. By \eqref{section4_p} and \eqref{solutiondq3}, we have 
\begin{align}\label{p1R1}
p_{1n}^{m}(R') &= \frac{\sqrt{n(n+1)}}{R'}g_n^m(R')-\frac{1}{R'} q_{2n}^m(R')-q_{2n}^{m'}(R')\notag\\
&= \frac{\sqrt{n(n+1)}}{R'}g_n^m(R')-\frac{1}{R'} q_{2n}^m(R')-\frac{1}{R'}z_n^{(2)}(\kappa_s R') q_{2n}^m(R')
-\frac{2\kappa_s}{\pi R'}\int_{R'}^R t^2 S_n^s(t)\hat{Z}_{2n}^m(t){\rm d}t\notag\\
&=  \frac{\sqrt{n(n+1)}}{R'}g_n^m(R')-\left[1+z_n^{(2)}(\kappa_s R')\right]\frac{1}{R'} q_{2n}^m(R')
-\frac{2\kappa_s}{\pi R'}\int_{R'}^R t^2 S_n^s(t)\hat{Z}_{2n}^m(t){\rm d}t.
\end{align}
A simple calculation from \eqref{section4_p}--\eqref{q1q3} yields 
\begin{equation}\label{p2R1}
p_{2n}^m(R') = -\frac{1}{\sqrt{n(n+1)}}\left[\kappa_s^2 R' q_{3n}^m(R')+R' \hat{Z}_{3n}^{m}(R')\right]. 
\end{equation}
It follows from \eqref{section4_p} and \eqref{appendix_g'R'} that
\begin{align}\label{p3R1}
p_{3n}^m(R') &= g_{n}^{m'}(R')-\frac{\sqrt{n(n+1)}}{R'}q_{2n}^m(R')\notag\\
&= \frac{1}{R'} z_n^{(2)}(\kappa_p R')g_n^m(R')+\frac{2\kappa_p}{\pi R'}\int_{R'}^R t^2 S^p_n(t)\hat{\zeta}_n^m(t){\rm d}t
-\frac{\sqrt{n(n+1)}}{R'}q_{2n}^m(R').
\end{align}
We get from \eqref{p1R1}--\eqref{p3R1} that
\begin{equation}\label{pR1}
\left(p_{1n}^m(R'), p_{2n}^m(R'), p_{3n}^m(R')\right)^\top=\overline{K_n(R')}\left(g_n^m(R'), q_{2n}^m(R'), q_{3n}^m(R')\right)^\top+\left(d_{1n}^m, d_{2n}^m, d_{3n}^m\right)^\top,
\end{equation}
where
\begin{align}\label{3Ddual_dnm}
\boldsymbol{d}_n^m =\left(d_{1n}^m, d_{2n}^m, d_{3n}^m\right)^\top&=\bigg(-\frac{2\kappa_s}{\pi R'}\int_{R'}^R t^2 S_n^s(t)\hat{Z}_{2n}^m(t){\rm d}t,\notag\\
&\quad \, -\frac{1}{\sqrt{n(n+1)}}R' \hat{Z}_{3n}^{m}(R'),\, \frac{2\kappa_p}{\pi R'}\int_{R'}^R t^2 S^p_n(t)\hat{\zeta}_n^m(t){\rm d}t\bigg)^\top. 
\end{align}
Substituting \eqref{pR1} into \eqref{pR} and using the definition \eqref{section4_uRR1}, we obtain 
\begin{eqnarray}\label{3Ddual_fpR}
	\boldsymbol{p}_n^m(R) 
	=\overline{Q_n}\,\boldsymbol{p}_n^m(R')
		-\overline{Q_n}\,\boldsymbol{d}_n^m+\overline{K_n(R)}\, \overline{M_{\rm diag}}\,\boldsymbol{b}_n^m.\\\notag
\end{eqnarray}

By \eqref{section4_zeta}, it can be verified that
\begin{align*}
b_{1n}^m &= \frac{1}{\lambda+2\mu} \frac{{\rm i}\kappa_p}{2}\int_{R'}^R t^2  W^p_n(R', t)\zeta_n^m(t){\rm d}t \\
&= \frac{1}{\lambda+2\mu} \frac{{\rm i}\kappa_p}{2}\frac{1}{2n+1}\int_{R'}^R t^2  W^p_n(R', t)\Bigg[\int_{t}^R\left[-n\left(\frac{t}{\tau}\right)^{-n-1}-(n+1) \left(\frac{t}{\tau}\right)^{n}\right]\xi_{3n}^m(\tau)\\
&\quad -\sqrt{n(n+1)}\left[\left(\frac{t}{\tau}\right)^{-n-1}-\left(\frac{t}{\tau}\right)^{n}\right]
\xi_{1n}^m(\tau){\rm d}\tau\Bigg]{\rm d}t. 
\end{align*}
We have from \eqref{section4_Z1nm} that 
\begin{align*}
b_{2n}^m &= \frac{{\rm i}\kappa_s}{2\mu}\int_{R'}^R t^2 W^s_n(R', t)Z_{2n}^m(t){\rm d}t \\
&= \frac{1}{2n+1}\frac{{\rm i}\kappa_s}{2\mu}\int_{R'}^R t^2 W^s_n(R', t)\left[\int_{t}^R\sqrt{n(n+1)}(c_2-c_4)\xi_{3n}^m(\tau)+\left[ (n+1)c_2+n c_4\right]\xi_{1n}^m(\tau){\rm d}\tau\right]{\rm d}t \\
&= \frac{1}{2n+1}\frac{{\rm i}\kappa_s}{2\mu}\int_{R'}^R t^2 W^s_n(R', t)\Bigg[\int_{t}^R\sqrt{n(n+1)}\left[ \left(\frac{t}{\tau}\right)^{-n-1}-\left(\frac{t}{\tau}\right)^{n}\right]\xi_{3n}^m(\tau)\\
&\quad + \left[ (n+1) \left(\frac{t}{\tau}\right)^{-n-1}+n \left(\frac{t}{\tau}\right)^{n}\right]\xi_{1n}^m(\tau)\,{\rm d}\tau \Bigg]{\rm d}t. 
\end{align*}
It follows from \eqref{section4_Z3nm} that
\begin{align*}
b_{3n}^m &= \frac{{\rm i}\kappa_s}{2\mu RR'} \int_{R'}^R t^3 W^s_n(R', t)Z_{3n}^m(t){\rm d}t \\
&= \frac{{\rm i}\kappa_s}{2\mu R R'} \frac{1}{2n+1} \int_{R'}^R t^3 W^s_n(R', t)\left[\int_{t}^R\sqrt{n(n+1)}(c_1-c_3) \xi_{2n}^m(\tau){\rm d}\tau\right]{\rm d}t\\
&=\frac{{\rm i}\kappa_s}{2\mu R R'} \frac{1}{2n+1} \int_{R'}^R t^3 W^s_n(R', t)\left[\int_{t}^R\sqrt{n(n+1)}\left[\left(\frac{t}{\tau}\right)^{-n-2}-\left(\frac{t}{\tau}\right)^{n-1}\right] \xi_{2n}^m(\tau){\rm d}\tau\right]{\rm d}t. 
\end{align*}
Substituting \eqref{section4_zeta} -- \eqref{section4_Z3nm} into \eqref{3Ddual_dnm}, we obtain 
\begin{align*}
d_{1n}^m &= -\frac{2\kappa_s}{\pi \mu R'}\int_{R'}^R t^2 S_n^s(t) Z_{2n}^m(t){\rm d}t \\
&=-\frac{2\kappa_s}{\pi \mu R'}\frac{1}{2n+1}\int_{R'}^R t^2 S_n^s(t)\Bigg[\int_{t}^R\sqrt{n(n+1)}\left[ \left(\frac{t}{\tau}\right)^{-n-1}-\left(\frac{t}{\tau}\right)^{n}\right]\xi_{3n}^m(\tau)\\
&\quad +\left[ (n+1) \left(\frac{t}{\tau}\right)^{-n-1}+n \left(\frac{t}{\tau}\right)^{n}\right]\xi_{1n}^m(\tau){\rm d}\tau
\Bigg]{\rm d}t, \\
d_{2n}^m &= -\frac{1}{\sqrt{n(n+1)}}\frac{R'}{\mu} Z_{3n}^{m}(R') \\
&= -\frac{1}{\sqrt{n(n+1)}}\frac{R'}{\mu} \frac{1}{2n+1}\int_{R'}^R\sqrt{n(n+1)}
\left[\left(\frac{R'}{\tau}\right)^{-n-2}-\left(\frac{R'}{\tau}\right)^{n-1}\right] \xi_{2n}^m(\tau){\rm d}\tau, \\
d_{3n}^m &= \frac{1}{\lambda+2\mu}\frac{2\kappa_p}{\pi R'}\int_{R'}^R t^2 S^p_n(t)\zeta_n^m(t){\rm d}t \\
&=\frac{1}{\lambda+2\mu}\frac{2\kappa_p}{\pi R'}\frac{1}{2n+1}\int_{R'}^R t^2 S^p_n(t)\Bigg[\int_{t}^R\left[-n\left(\frac{t}{\tau}\right)^{-n-1}-(n+1) \left(\frac{t}{\tau}\right)^{n}\right]\xi_{3n}^m(\tau)\\
&\quad -\sqrt{n(n+1)}\left[\left(\frac{t}{\tau}\right)^{-n-1}-\left(\frac{t}{\tau}\right)^{n}\right]\xi_{1n}^m(\tau){\rm d}\tau\Bigg]{\rm d}t.
\end{align*}

For sufficiently large $n$, it is shown in Lemma \ref{Firsttwo1} and  \cite{bzh-2020} that
\begin{eqnarray*}
\|Q_n\|\lesssim n\left(\frac{R'}{R}\right)^{n},\quad\|
\overline{K_n(R)}\overline{M_{\rm diag}}\|\lesssim n\left(\frac{R'}{R}\right)^{n},\quad 
\left|W_n(R', t)\right|\lesssim \frac{1}{n} \left(\frac{t}{R'}\right)^{n}.
\end{eqnarray*}
Substituting the above estimates into $\boldsymbol{b}_n^m$ and $\boldsymbol{d}_n^m$ leads to 
\[
\left|b_{jn}^m\right|\lesssim \frac{1}{n^2}\left(\frac{R}{R'}\right)^{n}\|\boldsymbol{\xi}\|_{L^{\infty}([R', R])},\quad
\left|d_{jn}^m\right|\lesssim \frac{1}{n^2} \left(\frac{R}{R'}\right)^{n}\|\boldsymbol{\xi}\|_{L^{\infty}([R', R])}.
\]
Hence we have from \eqref{3Ddual_fpR} that 
\begin{align*}
\left|\boldsymbol{p}_n^m(R)\right| &\leq \left|\overline{Q_n}\boldsymbol{p}_n^m(R')\right|+\left|\overline{Q_n}\boldsymbol{d}_n^m\right|+\left|\overline{K_n(R)}\overline{M_{\rm diag}}\boldsymbol{b}_n^m\right|\\
&\lesssim n\left(\frac{R'}{R}\right)^{n} \sum\limits_{i=1}^3\left|p_{in}^m(R')\right|
+\frac{1}{n}\|\boldsymbol{\xi}\|_{L^{\infty}([R', R])}+\frac{1}{n}\|\boldsymbol{\xi}\|_{L^{\infty}([R', R])},
\end{align*}
which completes the proof.
\end{proof}

\begin{lemma}\label{3Ddual_mainlemma}
Let $\boldsymbol{p}$ be the solution of the dual problem \eqref{section4_dual1problem}. For sufficiently large $N$,
the following estimate holds:
\[
\left|\int_{\Gamma_R}\left(\mathscr{T}-\mathscr{T}_N\right)\boldsymbol{\xi}\cdot\overline{\boldsymbol{p}}\right|\,{\rm d}s
\lesssim \frac{1}{N}\|\boldsymbol{\xi}\|_{\boldsymbol{H}^1(\Omega)}^2.
\]
\end{lemma}

\begin{proof}
It is shown in \cite{LY-ipi} that all the elements of the DtN matrix $M_n$ have an order $n$. Hence, 
	\begin{align}\label{3Ddual_finallemma}
		&\left|\int_{\Gamma_R}\left(\mathscr{T}-\mathscr{T}_N\right)\boldsymbol{\xi}\cdot\overline{\boldsymbol{p}}\,
			{\rm d}s\right|\leq \left|\sum\limits_{n=N+1}\sum\limits_{|m|\leq n} 
			M_{n} \boldsymbol{\xi}(R)\cdot \overline{\boldsymbol{p}}(R)\right|	\notag\\
		&\leq \sum\limits_{n=N+1}\sum\limits_{|m|\leq n} n \left(\left|\xi_{1n}^m(R)\right|+\left|\xi_{2n}^m(R)\right|
			+\left|\xi_{3n}^m(R)\right|\right) 
			\overline{\left(\left|p_{1n}^m(R)\right|+\left|p_{2n}^m(R)\right|+\left|p_{3n}^m(R)\right|\right)}\notag\\
		&\lesssim\sum\limits_{n=N+1} \left[n\left(1+n(n+1)\right)^{1/2}\right]^{-1/2}
			\left[\sum\limits_{|m|\leq n} \left(1+n(n+1)\right)^{1/2} \left(\left|\xi_{1n}^m(R)\right|^2
			+\left|\xi_{2n}^m(R)\right|^2+\left|\xi_{3n}^m(R)\right|^2\right)\right]^{1/2}\notag\\
		&\quad\times\left[\sum\limits_{|m|\leq n} n^3  \left(\left|p_{1n}^m(R)\right|^2
			+\left|p_{2n}^m(R)\right|^2+\left|p_{3n}^m(R)\right|^2\right)\right]^{1/2} \notag\\
		&\lesssim \frac{1}{N} \|\boldsymbol{\xi}\|_{\boldsymbol{H}^{1/2}(\Gamma_R)}	
			\left[\sum\limits_{n=N+1}\sum\limits_{|m|\leq n} n^3  \left(\left|p_{1n}^m(R)\right|^2
			+\left|p_{2n}^m(R)\right|^2+\left|p_{3n}^m(R)\right|^2\right)\right]^{1/2}\notag\\
		&\lesssim \frac{1}{N} \|\boldsymbol{\xi}\|_{\boldsymbol{H}^1(\Omega)}	
			\left[\sum\limits_{n=N+1}\sum\limits_{|m|\leq n} n^3  \left(\left|p_{1n}^m(R)\right|^2
			+\left|p_{2n}^m(R)\right|^2+\left|p_{3n}^m(R)\right|^2\right)\right]^{1/2}. 
	\end{align}
It follows from Lemma \ref{3Ddual_estimatep} that
	\begin{align*}
		&\sum\limits_{n=N+1}\sum\limits_{|m|\leq n} n^3  \left(\left|p_{1n}^m(R)\right|^2
			+\left|p_{2n}^m(R)\right|^2+\left|p_{3n}^m(R)\right|^2\right)\\
		&\lesssim \sum\limits_{n=N+1}\sum\limits_{|m|\leq n} n^3  
			\left[ n^2\left(\frac{R'}{R}\right)^{2n}\sum\limits_{i=1}^3|p_{in}^m(R')|^2
			+\frac{1}{n^2}\|\boldsymbol{\xi}\|^2_{L^{\infty}([R', R])}\right]\\
		& \lesssim \sum\limits_{n=N+1}\sum\limits_{|m|\leq n} n^5 \left(\frac{R'}{R}\right)^{2n}
			\sum\limits_{i=1}^3|p_{in}^m(R')|^2
			+\sum\limits_{n=N+1}\sum\limits_{|m|\leq n} (1+n(1+n))^{1/2} \|\boldsymbol{\xi}\|^2_{L^{\infty}([R', R])}\\
		&:= J_1+J_2. 
	\end{align*}
Noting that the function $t^4 e^{-2t}$ is bounded on $(0, +\infty),$ we have 
	\[
		J_1 \lesssim \max\limits_{n=N+1}\left(n^4 \left(\frac{R'}{R}\right)^{2n}\right)\sum\limits_{n=N+1}\sum\limits_{|m|\leq n}
		n \sum\limits_{i=1}^3|p_{in}^m(R')|^2\lesssim \|\boldsymbol{p}\|^2_{H^{1/2}(\Gamma_{R'})}
		\lesssim \|\boldsymbol{\xi}\|_{\boldsymbol{H}^1(\Omega)}.
	\]
It is shown in \cite{jllz-2017-ccp} that
\begin{align*}
\|\xi_{jn}^m(t)\|_{L^{\infty}([R', R])}^2 &\leq \left(\frac{2}{R-R'}+n\right)\|\xi_{jn}^m(t)\|_{L^2([R', R])}^2
+\frac{1}{n}\|\xi_{jn}^{m'}(t)\|_{L^2([R', R])}^2,\\
\|\xi_{jn}^m\|^2_{H^1(B_R\backslash B_{R'})} &\geq \int_{R'}^R \left[\left(R'+\frac{n^2}{R}\right)|\xi_{jn}^m(\rho)|^2
+R' \left|\xi_{jn}^{m'}(\rho)\right|^2\right]{\rm d}\rho. 
\end{align*}
Substituting them into $J_2$ gives
	\begin{align*}
		J_2	 &\lesssim \sum\limits_{n=N+1}\sum\limits_{|m|\leq n} (1+n(1+n))^{1/2}
			\sum\limits_{j=1}^3 \left[ \left(\frac{2}{R-R'}+n\right)\|\xi_{jn}^m(t)\|_{L^2([R', R])}^2
			+\frac{1}{n}\|\xi_{jn}^{m'}(t)\|_{L^2([R', R])}^2\right]\\
		&\lesssim \|\boldsymbol{\xi}\|^2_{\boldsymbol{H}^1(B_R\backslash B_{R'})}
			\leq \|\boldsymbol{\xi}\|_{\boldsymbol{H}^1(\Omega)}^2,
	\end{align*}
	which yields 
	\[
		\sum\limits_{n=N+1}\sum\limits_{|m|\leq n} n^3  \left(\left|p_{1n}^m(R)\right|^2
			+\left|p_{2n}^m(R)\right|^2+\left|p_{3n}^m(R)\right|^2\right)
			\lesssim \|\boldsymbol{\xi}\|_{\boldsymbol{H}^1(\Omega)}^2.
	\]
Plugging the above estimate into \eqref{3Ddual_finallemma}, we obtain 
	\[
		\left|\int_{\Gamma_R}\left(\mathscr{T}-\mathscr{T}_N\right)\boldsymbol{\xi}\cdot\overline{\boldsymbol{p}}
			{\rm d}s\right|\lesssim \frac{1}{N} \|\boldsymbol{\xi}\|_{\boldsymbol{H}^1(\Omega)}^2,
	\]
	which completes the proof.
\end{proof}

We are now in position to show the proof of Theorem \ref{Mainthm}. 

\begin{proof}
It follows from \eqref{section4_key} that
\begin{align*}
	\vvvert\boldsymbol{\xi}\vvvert^2_{\boldsymbol{H}^{1}(\Omega)} &=
		\Re b(\boldsymbol{\xi}, \boldsymbol{\xi})+\Re\int_{\Gamma_R}
		\left(\mathscr{T}-\mathscr{T}_N\right)\boldsymbol{\xi}\cdot\overline{\boldsymbol{\xi}}{\rm d}s 
		+2\omega^2 \int_{\Omega}\boldsymbol{\xi}\cdot\overline{\boldsymbol{\xi}}{\rm d}\boldsymbol{x}
		+\Re\int_{\Gamma_R} \mathscr{T}_N\boldsymbol{\xi}\cdot\overline{\boldsymbol{\xi}}{\rm d}s\\	
	&\leq C_1\left[\left(\sum\limits_{T\in M_h} \eta_T^2\right)^{1/2}
		+N\left(\frac{R'}{R}\right)^N \|\boldsymbol u^{\rm inc}\|_{\boldsymbol{H}^1(\Omega)}
		+\|\boldsymbol{g}-\boldsymbol{g}^h\|_{\boldsymbol{H}^{1/2}(\partial D)}\right] 	
		\|\boldsymbol{\xi}\|_{\boldsymbol H^{1}(\Omega)}\\
	&\quad +\left(C_2+C(\delta)\right)\|\boldsymbol{\xi}\|^2_{\boldsymbol{L}^2(\Omega)} 
		+\left(\frac{R}{R'}\right)\delta\|\boldsymbol{\xi}\|^2_{\boldsymbol H^{1}(\Omega)}.	
\end{align*}		
Choosing $\delta$ such that $\frac{R}{\hat R}\frac{\delta}{\min(\mu, \omega^2)}<\frac{1}{2}$, we get
\begin{align}\label{Estimate_H1}
	\vvvert\boldsymbol{\xi}\vvvert^2_{\boldsymbol{H}^{1}(\Omega)}& \leq 
		2C_1\left[\left(\sum\limits_{T\in M_h} \eta_T^2\right)^{1/2}
		+N\left(\frac{R'}{R}\right)^N \|\boldsymbol u^{\rm inc}\|_{\boldsymbol{H}^1(\Omega)}
		+\|\boldsymbol{g}-\boldsymbol{g}^h\|_{\boldsymbol{H}^{1/2}(\partial D)}\right] 
		\|\boldsymbol{\xi}\|_{\boldsymbol H^{1}(\Omega)}\notag\\
&\quad 	+2\left(C_2+C(\delta)\right)\|\boldsymbol{\xi}\|^2_{\boldsymbol{L} ^2(\Omega)}.
\end{align}
Using Lemmas \ref{Poincare}, \ref{Firsttwo2}, and \ref{3Ddual_mainlemma} yields  
\begin{align*}
\|\boldsymbol{\xi}\|^{2}_{\boldsymbol{L}^2(\Omega)}&= b(\boldsymbol{\xi}, \boldsymbol{p})+\int_{\Gamma_R}\left(\mathscr{T}-\mathscr{T}_N\right)\boldsymbol{\xi}\cdot\overline{\boldsymbol{p}}{\rm d}s-\int_{\Gamma_R}\left(\mathscr{T}-\mathscr{T}_N\right)\boldsymbol{\xi}\cdot\overline{\boldsymbol{p}}{\rm d}s\\
&\lesssim \left[\left(\sum\limits_{T\in M_h} \eta_T^2\right)^{1/2}+N\left(\frac{R'}{R}\right)^N \|\boldsymbol u^{\rm inc}\|_{\boldsymbol{H}^1(\Omega)}+\|\boldsymbol{g}-\boldsymbol{g}^h\|_{\boldsymbol{H}^{1/2}(\partial D)}\right] 	
\|\boldsymbol{\xi}\|_{H^{1}(\Omega)}+\frac{1}{N}
\|\boldsymbol{\xi}\|^2_{H^{1}(\Omega)}.
\end{align*}
Substituting the above estimate into \eqref{Estimate_H1} and taking sufficiently large $N$ such that 
\[
\frac{2\left(C_2+C(\delta)\right)}{N}\frac{1}{\min(\mu, \omega^2)}<1,
\]
we obtain 
\begin{eqnarray*}
\vvvert\boldsymbol{u}-\boldsymbol{u}_N^h
\vvvert_{\boldsymbol{H}^{1}(\Omega)}\lesssim 
\left(\sum\limits_{T\in M_h} \eta_T^2\right)^{1/2}
+\|\boldsymbol{g}-\boldsymbol{g}^h\|_{\boldsymbol{H}^{1/2}(\partial D)}
		+N\left(\frac{R'}{R}\right)^N \|\boldsymbol u^{\rm inc}\|_{\boldsymbol{H}^1(\Omega)},
\end{eqnarray*}
which completes the proof. 
\end{proof}

\section{Implementation and numerical experiments}\label{Section_ne}

In this section, we discuss the algorithmic implementation of the adaptive
finite element DtN method and present two numerical examples to demonstrate the
effectiveness of the proposed method. 

\begin{table}
\caption{The adaptive finite element DtN method.}
\hrule \hrule
\vspace{0.8ex}
\begin{enumerate}		
\item Given the tolerance $\epsilon>0, \theta\in(0,1)$;
\item Fix the computational domain $\Omega=B_R\setminus \overline{D}$ by
choosing the radius $R$;
\item Choose $R'$ and $N$ such that $\epsilon_N\leq 10^{-8}$;
\item Construct an initial mesh $\mathcal M_h$ over $\Omega$ and
compute error estimators;
\item While $\epsilon_h>\epsilon$ do
\item \qquad Refine the mesh $\mathcal M_h$ according to the strategy:
\[
\text{if } \eta_{\hat{T}}>\theta \max\limits_{T\in \mathcal M_h}
\eta_{T}, \text{ then refine the element } \hat{T}\in\mathcal M_h;
\]
\item \qquad Solve the discrete problem \eqref{section3_dis1} on the new mesh which is still denoted as $\mathcal M_h$;
\item \qquad Compute the corresponding error estimators;
\item End while.
\end{enumerate}
\vspace{0.8ex}
\hrule\hrule
\end{table}

\subsection{Adaptive algorithm}\label{algorithm}

Based on the a posteriori error estimate in Theorem \ref{Mainthm}, we adopt the FreeFem \cite{H-jnm12} to implement the adaptive algorithm of the linear finite elements. By Theorem \ref{Mainthm}, the a posteriori error estimator consists of three parts: the first two parts are related to the finite element discretization error $\epsilon_h$ and the third one is the DtN truncation error $\epsilon_N$ which depends on the truncation number $N$. Explicitly,
\begin{equation*}
\epsilon_h =  \left(\sum\limits_{T\in \mathcal{M}_h} \eta_T^2\right)^{1/2}+\|\boldsymbol{g}-\boldsymbol{g}^h\|_{\boldsymbol{H}^{1/2}(\partial D)}, \quad \epsilon_N =N \left(\frac{R'}{R}\right)^{N}\|\boldsymbol u^{\rm inc}\|_{\boldsymbol{H}^1(\Omega)}. 
\end{equation*}

In the implementation, the parameters $R', R$, and $N$ are chosen such that the finite element discretization  error is not polluted by the DtN truncation error, i.e., $\epsilon_N$ is required to be very small compared to $\epsilon_h$, for example, 
$\epsilon_N\leq 10^{-8}$. For simplicity, in the following numerical experiments, $R'$ is chosen such that the obstacle is exactly contained in the ball $B_{R'}$, and $N$ is taken to be the smallest positive integer such that $\epsilon_N\leq 10^{-8}$. Table 1 shows the algorithm of the adaptive finite element DtN method for solving the elastic obstacle scattering problem. 

\subsection{TBC matrix construction}
Denote by $\left\{\boldsymbol{\Psi}_j\right\}_{j=1}^{L}$ the basis of the finite element space $\boldsymbol{V}_h$. Then
the finite element approximation of the solution is
\begin{equation}\label{Uapprox}
	\boldsymbol{u}_N=\sum\limits_{j=1}^{L} u_j \boldsymbol{\Psi}_j.
\end{equation}

Recall that the truncated DtN operator \eqref{section3_truncDtN} is
\begin{align}\label{TBCN}
\mathscr{T}_N \boldsymbol{u}_N &= \sum\limits_{n=0}^{N}\sum\limits_{m=-n}^n \Bigg\{\left[M_{11}^{(n,m)} u_{1n}^m(R) +M_{13}^{(n, m)} u_{3n}^m (R)\right] \boldsymbol U_n^m+M_{22}^{(n,m)} u_{2n}^m(R)\boldsymbol V_n^m \notag \\
&\quad +\left[M_{31}^{(n,m)} u_{1n}^m(R) +M_{33}^{(n, m)} u_{3n}^m (R)\right] X_n^m \boldsymbol e_{\rho} \Bigg\},
\end{align}
where 
\begin{equation}\label{ucoef}
 \left(u_{1n}^m(R), u_{2n}^m(R), u_{3n}^m(R) \right)^\top=\int_{\Gamma_R}\boldsymbol u_N\cdot\left(\overline{\boldsymbol U_n^m}, \overline{\boldsymbol V_n^m},  \left(\overline{X_n^m} \boldsymbol e_{\rho}\right)\right)^\top{\rm d}s.
\end{equation}
Substituting \eqref{Uapprox} into \eqref{ucoef} yields 
\begin{equation}\label{uinm}
 \left(u_{1n}^m(R), u_{2n}^m(R), u_{3n}^m(R) \right)^\top=\sum\limits_{j=1}^{L} u_j \int_{\Gamma_R} \boldsymbol{\Psi}_j\cdot\left(\overline{\boldsymbol U_n^m}, \overline{\boldsymbol V_n^m}, \left(\overline{X_n^m}\boldsymbol  e_{\rho}\right) \right)^\top{\rm d}s. 
\end{equation}

Define the components of the vectors $\Phi_{U}^{(n, m)}, \Phi_{V}^{(n, m)},$ and $\Phi_{X}^{(n, m)} $ as follows
\begin{align*}
\Phi_{U,j}^{(n, m)}=\int_{\Gamma_R} \boldsymbol{\Psi}_j\cdot\boldsymbol U_n^m {\rm d}s, \quad
\Phi_{V,j}^{(n, m)}=\int_{\Gamma_R} \boldsymbol{\Psi}_j\cdot\boldsymbol V_n^m {\rm d}s, \quad
\Phi_{X,j}^{(n, m)}=\int_{\Gamma_R} \boldsymbol{\Psi}_j\cdot\left( X_n^m \boldsymbol e_{\rho} \right){\rm d}s.
\end{align*}
Denote by $B$ the TBC matrix. It follows from \eqref{TBCN} and \eqref{uinm} that
\begin{align*}
\sum\limits_{j=1}^{L} B_{ij} u_j &= \int_{\Gamma_R} \mathscr{T}_N\boldsymbol{u}_N\cdot \boldsymbol{\Psi}_i {\rm d}s\\
&=\sum\limits_{n=0}^{N}\sum\limits_{m=-n}^n \Bigg\{\left[M_{11}^{(n, m)} \sum\limits_{j} u_j \overline{\Phi_{U,j}^{(n, m)}}
+M_{13}^{(n, m)} \sum\limits_{j} u_j \overline{\Phi_{X,j}^{(n, m)}}\right]\Phi_{U,i}^{(n, m)}\\
&\quad	+\left[M_{31}^{(n, m)} \sum\limits_{j} u_j \overline{\Phi_{U,j}^{(n, m)}}
			+M_{33}^{(n, m)} \sum\limits_{j} u_j \overline{\Phi_{X,j}^{(n, m)}}\right]
			\Phi_{X,i}^{(n, m)}\\
	&\quad +M_{22}^{(n, m)} \sum\limits_{j} u_j \overline{\Phi_{V,j}^{(n, m)}} \Phi_{V,i}^{(n, m)}\Bigg\},
\end{align*}
which indicates that the TBC matrix is a low rank matrix constructed by vector products. Specifically, we have 
\begin{align}\label{Bmatrix}
B&= \sum\limits_{n=0}^{N}\sum\limits_{m=-n}^n  \Big( M_{11}^{(n, m)} \Phi_U^{(n, m)} \Phi_U^{(n, m)^*}
+M_{13}^{(n, m)} \Phi_U^{(n, m)} \Phi_X^{(n, m)^*}+M_{22}^{(n, m)} \Phi_V^{(n, m)} \Phi_V^{(n, m)^*}\notag\\
&\quad+M_{31}^{(n, m)} \Phi_X^{(n, m)} \Phi_U^{(n, m)^*}+M_{33}^{(n, m)} \Phi_X^{(n, m)} \Phi_X^{(n, m)^*}\Big).
\end{align}
To simplify the notation, define matrices $U$ and $V$  as
\begin{align*}
U&=\Big( \Phi_U^{(n, m)},\Phi_U^{(n, m)}, \Phi_V^{(n, m)}, \Phi_X^{(n, m)}, \Phi_X^{(n, m)} \Big)_{(n, m)}, \\
V&=\Big( M_{11}^{(n, m)}\Phi_U^{(n, m)^*}, M_{13}^{(n, m)}\Phi_X^{(n, m)^*}, M_{22}^{(n, m)}\Phi_V^{(n, m)^*}, M_{31}^{(n, m)}\Phi_U^{(n, m)^*}, M_{33}^{(n, m)}\Phi_X^{(n, m)^*} \Big)_{(n, m)}.
\end{align*}
Then we have from \eqref{Bmatrix} that
\begin{equation}\label{BUV}
B=UV.
\end{equation}

Denote by $A$ the stiffness matrix corresponding to the variational problem with the Neumann condition $\partial_{\nu} \boldsymbol{u}=0$ on $\Gamma_R$ and the Dirichlet boundary condition on $\partial D$, which has the sesquilinear form
\[
	\hat{b}(\boldsymbol{u}, \boldsymbol{v})=
		\mu\int_{\Omega}\nabla\boldsymbol{u}:\nabla\overline{\boldsymbol{v}}\,{\rm d}\boldsymbol{x}
	+(\lambda+\mu)\int_{\Omega}\left(\nabla\cdot\boldsymbol{u}\right)
	\left(\nabla\cdot\overline{\boldsymbol{v}}\right)\,{\rm d}\boldsymbol{x} \notag\\
	-\omega^2\int_{\Omega} \boldsymbol{u}\cdot\overline{\boldsymbol{v}}\,{\rm d}\boldsymbol{x}
\]
Then the stiffness matrix $W$ for the variational problem \eqref{section3_dis1} takes the form
\begin{equation}\label{StiffnessM}
W=A-B. 
\end{equation}

Since the DtN operator is nonlocal, it is clear to note from \eqref{BUV} that the nonzeros of $B$ and $A$
are $O(h^{-4})$ and $O(h^{-3})$, respectively. Figure \ref{figmatrix} plots the sparsity patterns of matrices
$A$ and $W$ with 1805 nodal points in the mesh. Hence, the bandwidth of matrix $W$ is much larger than the bandwidth of matrix $A$. For the above reason, we do not assemble matrix $W$ and solve the resulting linear system directly.

The matrix $A$ is sparse, real and symmetric. It can be handled effectively by many parallel direct solvers. Since $B$ is a low rank matrix given in \eqref{BUV}, the linear system $Wz=b$ can be solved effectively via the generalized
Woodbury matrix identity by using the following steps: (1) construct matrices $U$ and $V$; (2) solve $A z_1=b$
with a parallel direct solver; (3) perform matrix-vector product $z_2=V z_1$; (4) construct matrix $C=A^{-1} U$
and $H=I+VC$; (5) solve a dense but small size linear system $H z_3=z_2$; (6) perform matrix-vector product
$z_4=Cz_3$; (7) construct the solution of the system $Wz=b$ by $z=z_1-z_4$. The above approach has several advantages:
(1) it is only needed to store matrices $A$,  $U$ and $V$. The nonzeros of $U, V$ are $O(h^{-2})$; (2) the bandwidth of matrix $A$ is much smaller, so it is faster to solve $Az_1=b$. Once the linear solver is set up, the construction of matrix $C=A^{-1} U$ is very fast.

\begin{figure}
\centering
\includegraphics[width=0.4\textwidth]{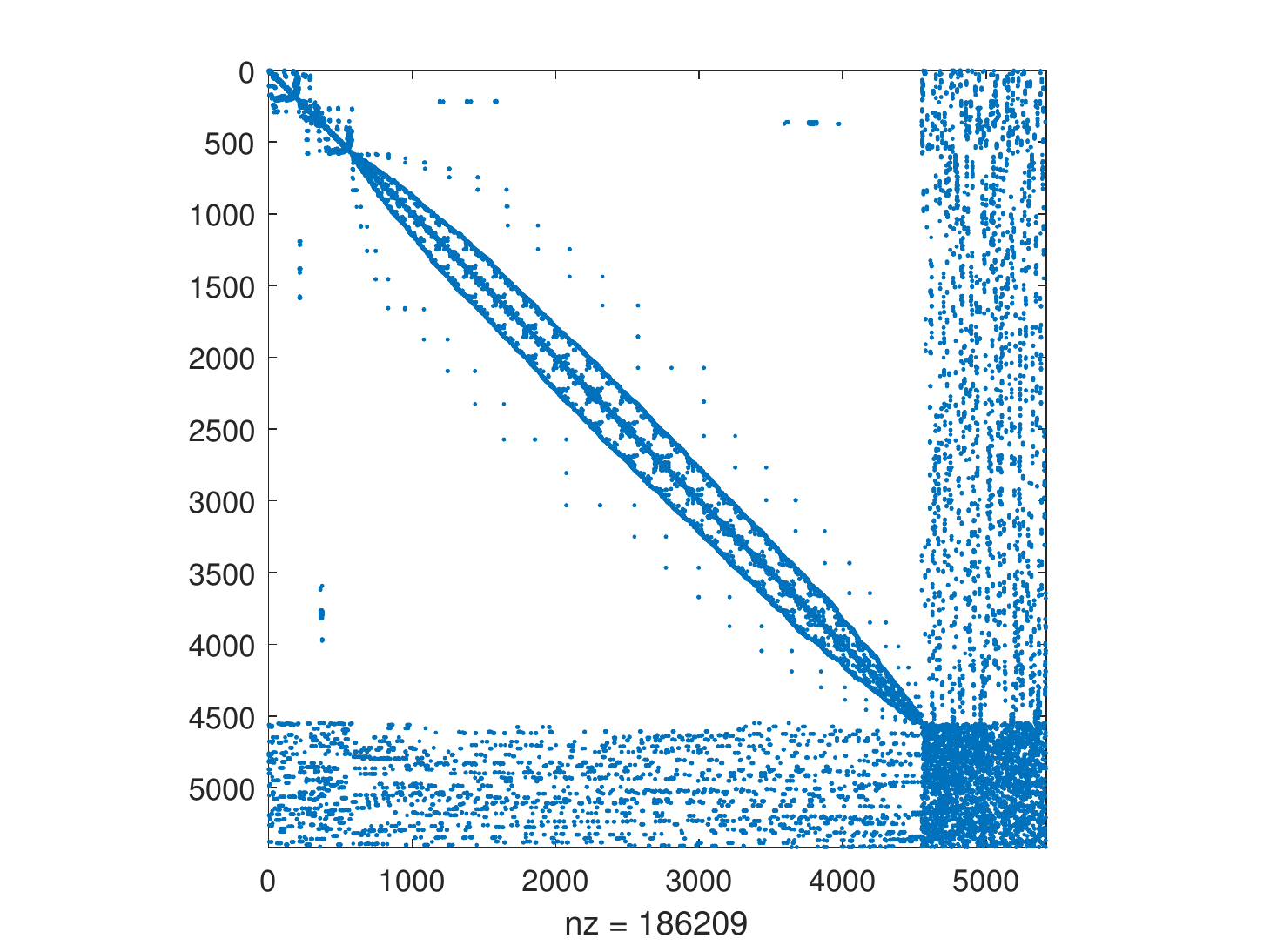}
\includegraphics[width=0.4\textwidth]{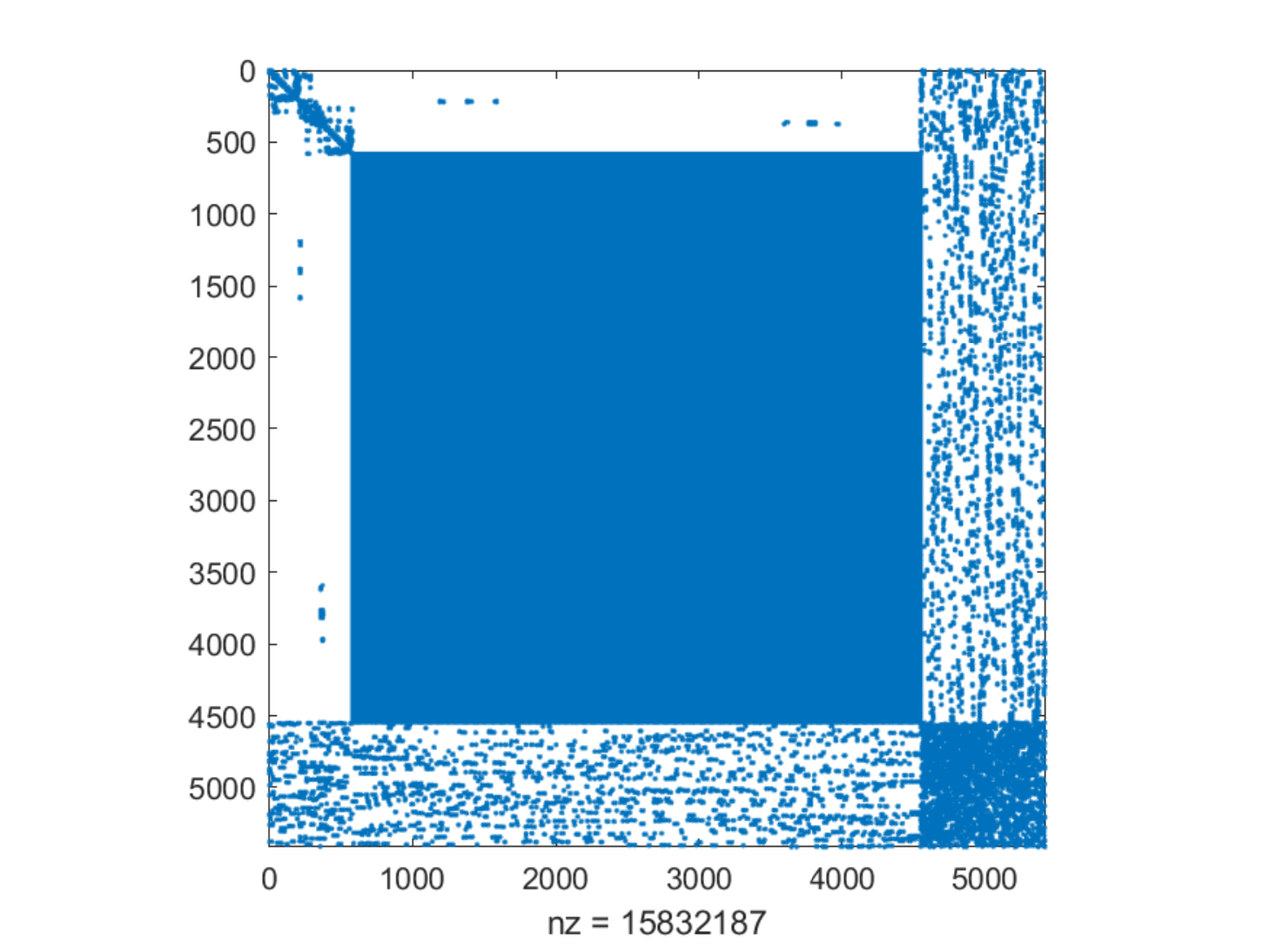}
\caption{Sparsity pattern of the coefficient matrix. (left) Matrix $A$; (right) Matrix $\mathscr{M}$.}
\label{figmatrix}
\end{figure}

\subsection{Numerical examples}

We present two examples to show the performance of the proposed method. In the first example, the obstacle is a ball so that the explicit solution is available. The incident wave is chosen as the third column of the Green Tensor. By comparing the numerical solution with the explicit solution, we are able to report the accuracy of the proposed algorithm. In the second example, we consider a more complex geometry: a rectangular U-shaped obstacle. The obstacle is assumed to be illuminated by a compressional plane wave. We pay particular attention on the mesh refinement around the corners of the obstacle, where the solution has singularity. In both examples, the a posteriori error is plotted against the number of unknowns in order to show the convergence rate.

{\em Example 1}. In this example, we intend to test the accuracy of the proposed algorithm. The obstacle is taken as a ball with radius $0.5$. The transparent boundary condition is set on the sphere $\Gamma_R$ with radius $R=1$.
Denote by $\boldsymbol{G}(\boldsymbol x, \boldsymbol y; \omega)$ the Green Tensor of the three-dimensional elastic wave equation. More explicitly, we have 
\[
\boldsymbol{G}(\boldsymbol x, \boldsymbol y; \omega)=\frac{1}{\mu}g(\boldsymbol x, \boldsymbol y; \kappa_s) \boldsymbol{I}+\frac{1}{\omega^2}\nabla_{\boldsymbol x}\nabla^\top_{\boldsymbol x}\left(g(\boldsymbol x, \boldsymbol y; \kappa_s)-g(\boldsymbol x, \boldsymbol y; \kappa_p)\right),
\]
where $\boldsymbol{I}$ is the $3\times 3$ identity matrix and
\[
g(\boldsymbol x, \boldsymbol y; \kappa)=\frac{1}{4\pi} \frac{e^{{\rm i}\kappa |\boldsymbol x-\boldsymbol y|}}{|\boldsymbol x-\boldsymbol y|}
\]
is the fundamental solution of the three-dimensional Helmholtz equation. The incident wave is chosen as a multiple of the third column of Green tensor, i.e., 
\[
\boldsymbol{u}^{\rm inc}(\boldsymbol x; \boldsymbol y)=10 \boldsymbol{G}(\boldsymbol x, \boldsymbol y; \omega)(:, 3),
\]
where the source point $\boldsymbol y=(0,0,0)^\top$ is taken as the origin, the angular frequency $\omega=\pi$, and the Lam\'{e} parameters $\mu=1, \lambda=2.$ Then it is easy to check that the scattered field is $\boldsymbol{u}=-\boldsymbol{u}^{\rm inc}$.

Denote the a priori error by $e_h=\|\boldsymbol{u}-\boldsymbol{u}_N^h\|_{\boldsymbol{H}^1(\Omega)}$.
Figure \ref{fig2} displays the curves of $\log e_h$ and $\log \epsilon_h$ against $\log \text{DoF}_h$ for the adaptive mesh refinements, where ${\rm DoF}_h$ stands for the degrees of freedom or the number of knowns for the mesh $\mathcal M_h$. The red circle line represents the a posteriori error estimate and the yellow star line stands for the a priori error estimate; the slope of the straight blue line is $-1/3$. It indicates that the meshes and the associated numerical complexity are quasi-optimal,
i.e., $\|\boldsymbol{u}-\boldsymbol{u}_N^h\|_{\boldsymbol{H}^1(\Omega)}
=O\big(\text{DoF}_h^{-1/3}\big)$ holds asymptotically.

\begin{figure}
\centering
\includegraphics[width=0.4\textwidth]{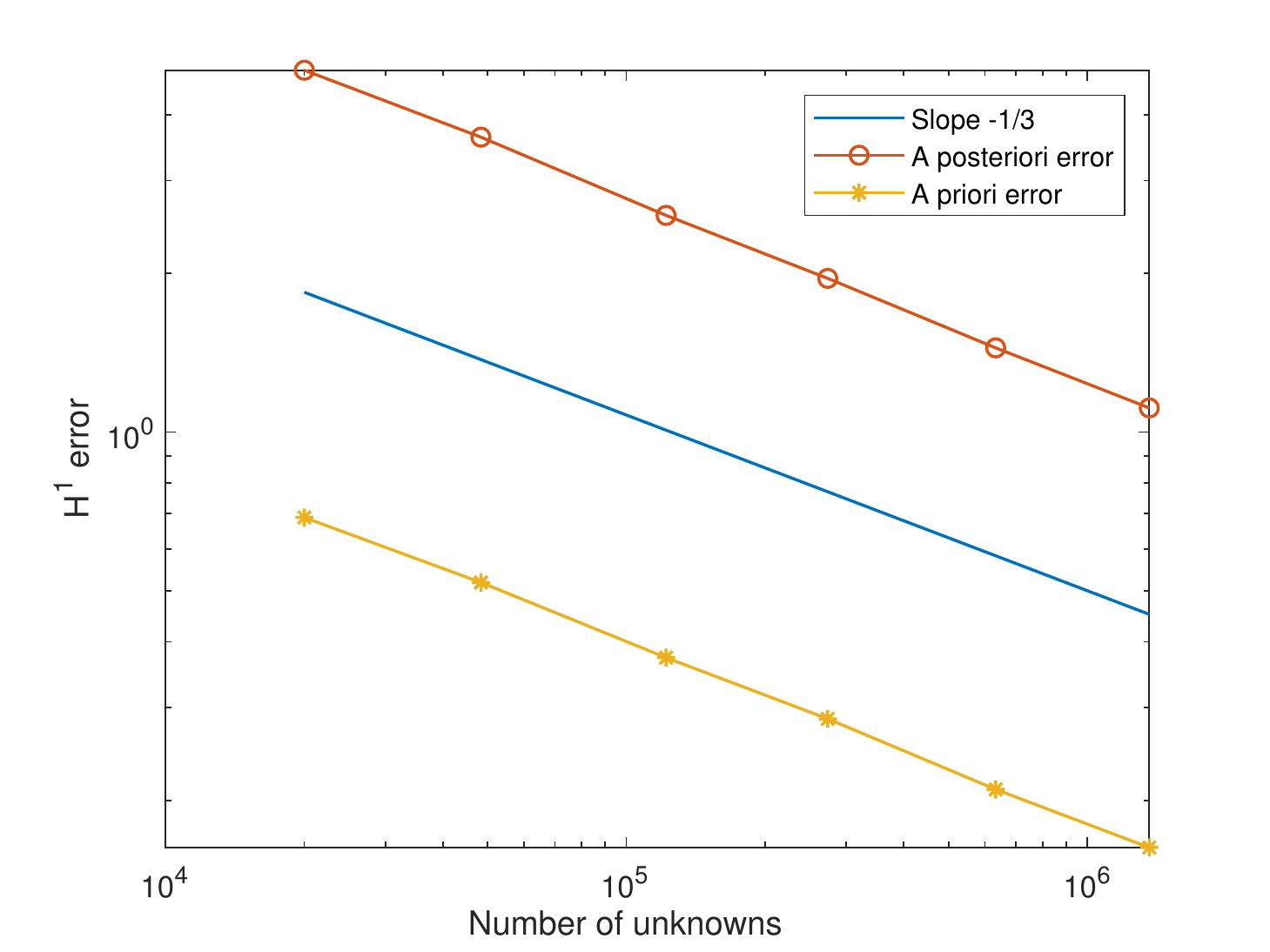}
\caption{Example 1. Quasi-optimality of the a priori and a posteriori error estimates.}
\label{fig2}
\end{figure}

{\em Example 2.} This example demonstrates the effectiveness of the proposed method to handle the problem where the solution has singularity. We consider a more complex geometry: a rectangular U-shaped obstacle, which is shown in the left of Figure 3. In this example, the parameters of the rectangular U-shaped geometry are $L=1, D=0.6, H=W=0.2$. The obstacle is assumed to be illuminated by a compressional plane wave
\[
\boldsymbol{u}^{\rm inc}(\boldsymbol x;\boldsymbol{d})= \boldsymbol{d} e^{{\rm i}\kappa_p \boldsymbol x\cdot\boldsymbol{d}},
\]
where $\boldsymbol{d}=(0,-1, 0)^\top$ is the incident direction, the angular frequency $\omega=\pi$, and the lam\'{e} parameters $\mu=1, \lambda=2$. Thus the compressional wavenumber $\kappa_p=\pi/2.$ The transparent boundary condition
is set on the sphere $\Gamma_R$ with radius $R=1$. Figure \ref{fig4} plots the curves of $\log \epsilon_h$ versus $\log \text{DoF}_h$ for the adaptive mesh refinements. The blue circle line represents for the a posteriori error estimate and the red line is a straight line with slope $-1/3$.  Again, it indicates that the meshes and the associated numerical complexity are quasi-optimal. Figure \ref{fig5} shows the initial mesh (1805 nodal points) and the refined mesh after 2 iterative steps with 13352 nodal points. It is clear to note that the mesh is refined mainly around the corners and the interior part of the U-shaped obstacle, where the solution has singularity, and keeps relatively coarse near the TBC surface, where the solution is smooth. 

\begin{figure}
\centering
\includegraphics[width=0.3\textwidth]{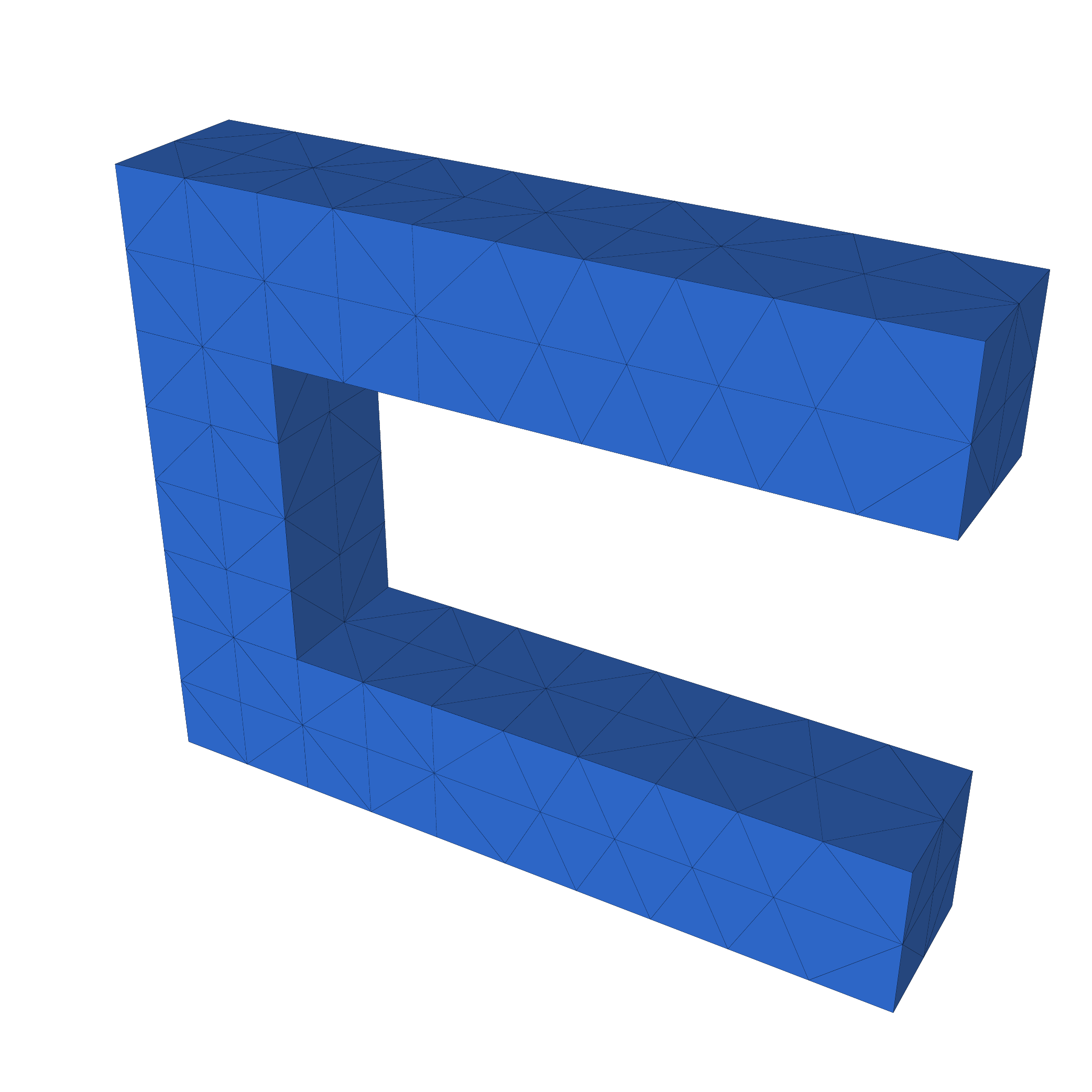}
\includegraphics[width=0.4\textwidth]{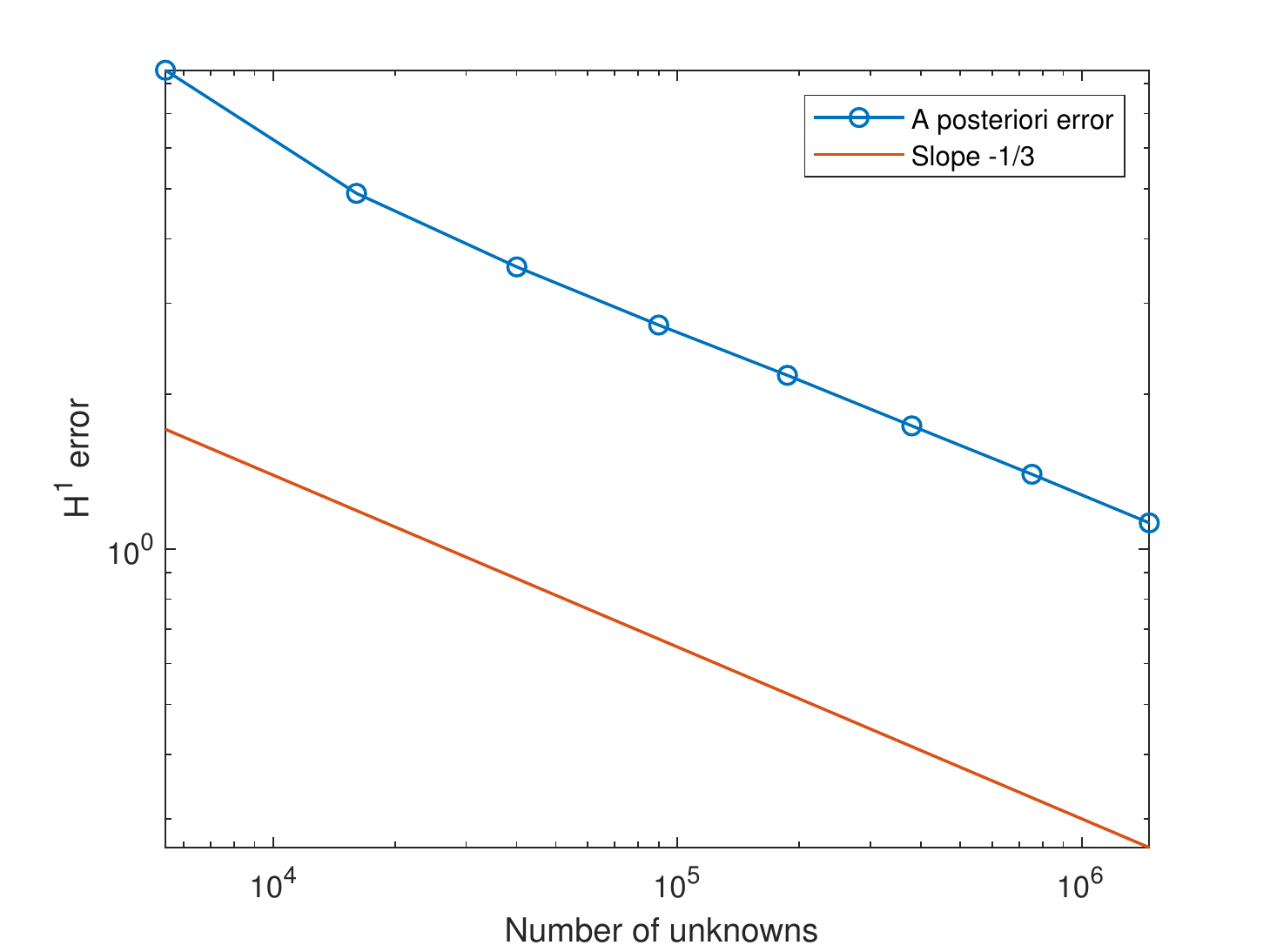}
\caption{Example 2: Quasi-optimality of the a posteriori error estimates.}
\label{fig4}
\end{figure}

\begin{figure}
\centering
\includegraphics[width=0.35\textwidth]{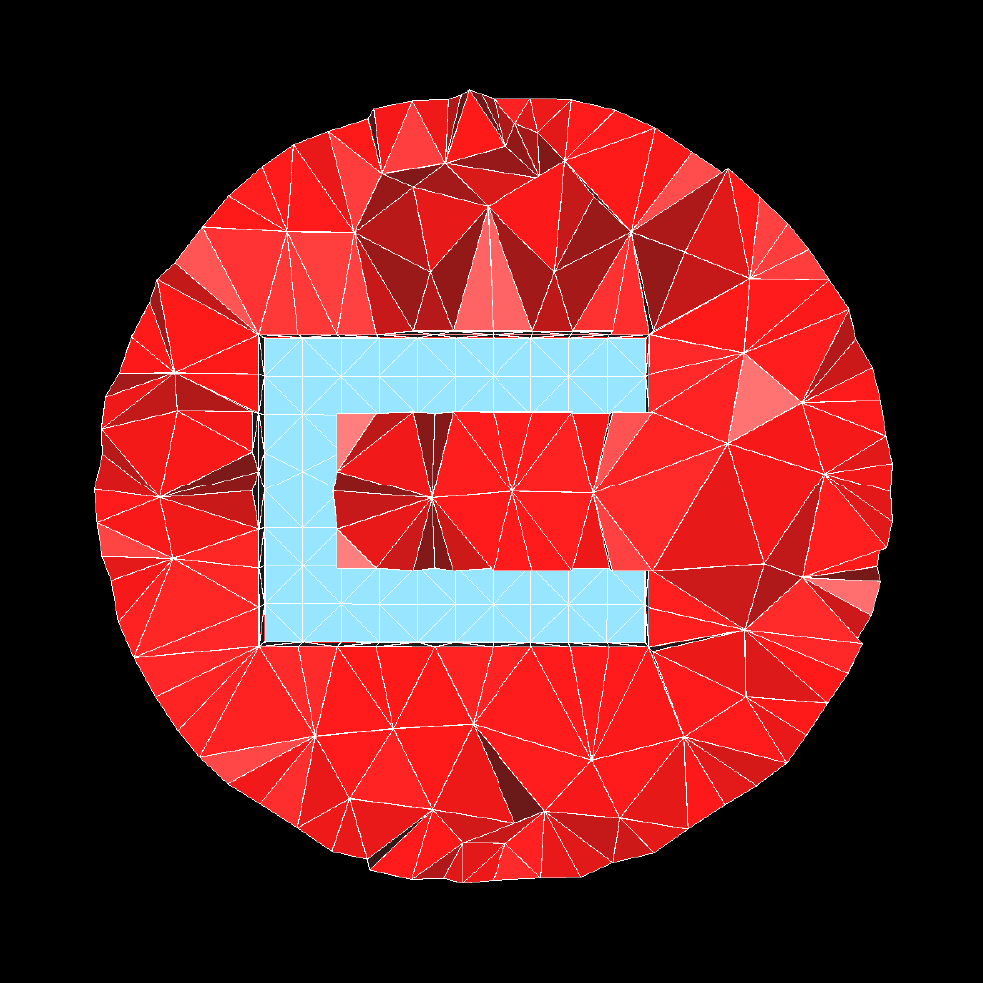}
\includegraphics[width=0.35\textwidth]{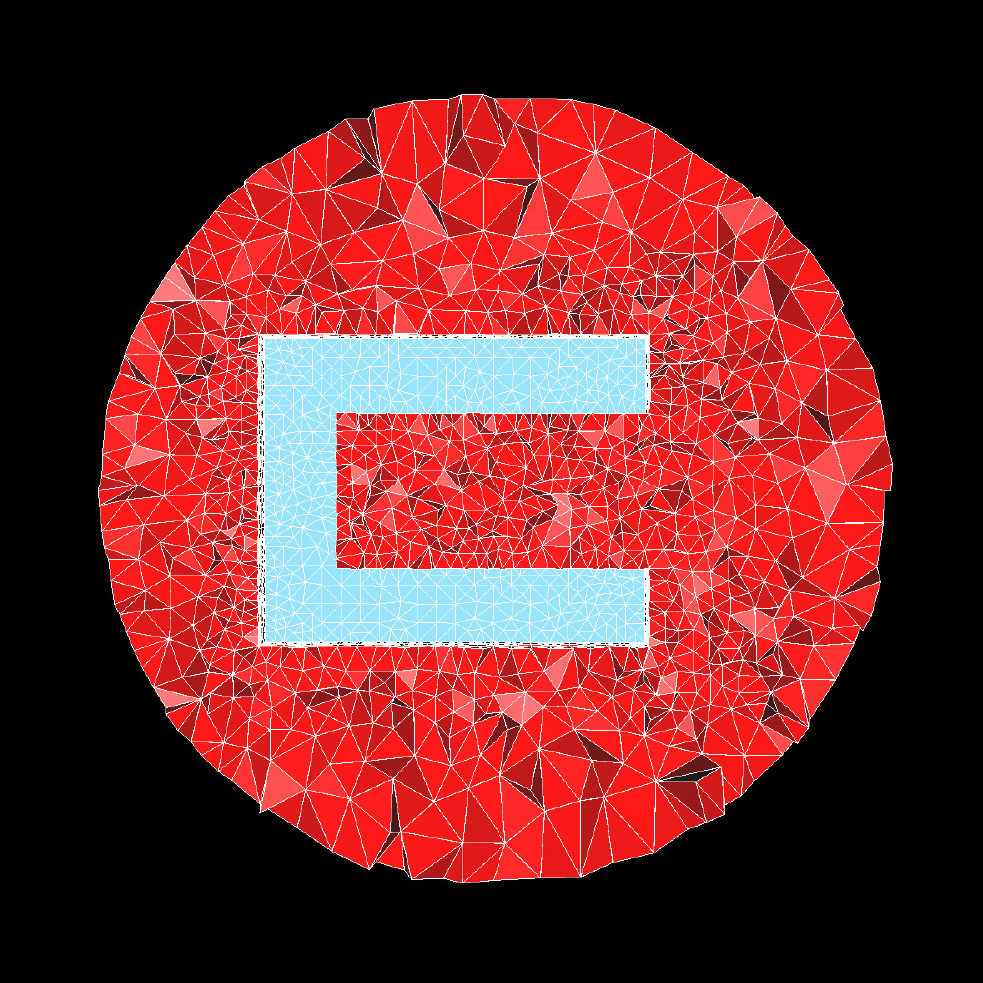}
\caption{Example 2: Cross sections of the initial mesh (left) and refined mesh (right).}
\label{fig5}
\end{figure}

\section{Conclusion}\label{Section_c}

In this paper, we have presented an adaptive finite element DtN method for the elastic obstacle scattering problem in three dimensions. Based on the Helmholtz decomposition, a new duality argument is developed to obtain the a posteriori error estimate. It not only takes into account of the finite element discretization error but also includes the truncation error of the DtN operator. We show that the truncation error decays exponentially with respect to the truncation parameter. The a posteriori error estimate for the solution of the discrete problem serves as a basis for the adaptive finite element approximation. Numerical results show that the proposed method is accurate and effective. 

\appendix

\section{Basis functions}\label{appendixbasis}

The spherical coordinates $(\rho,\theta,\varphi)$ are related to the Cartesian coordinates $\boldsymbol{x}=(x_1,x_2,x_3)$ by
$x_1=\rho\sin\theta\cos\varphi$, $x_2=\rho\sin\theta\sin\varphi$, $x_3=\rho\cos\theta$,  with the local orthonormal basis $\{\boldsymbol e_{\rho}, \boldsymbol e_{\theta}, \boldsymbol e_{\varphi}\}$:
\begin{align*}
\boldsymbol e_{\rho}&=(\sin\theta\cos\varphi,\sin\theta\sin\varphi,\cos\theta)^\top,\\
\boldsymbol e_{\theta}&=(\cos\theta\cos\varphi,\cos\theta\sin\varphi,-\sin\theta)^\top,\\
\boldsymbol e_{\varphi}&=(-\sin\varphi,\cos\varphi,0)^\top,
\end{align*}
where $\theta\in [0,\pi]$ and $\varphi\in [0,2\pi]$ are the Euler angles.

Denote by $\{Y_n^m(\theta, \varphi): |m|\leq n, n=0, 1, 2, \dots\}$ the orthonormal sequence of spherical harmonics of order $n$ on the unit sphere $\mathbb S^2=\{\boldsymbol x\in\mathbb R^3: |\boldsymbol x|=1\}$. Explicitly, they are given by 
\begin{equation*}
Y_{n}^{m}(\theta,\varphi)
=\sqrt{\frac{(2n+1)(n-|m|)!}{4\pi(n+|m|)!}}P_n^{|m|}(\cos\theta)
e^{{\rm i}m\varphi},
\end{equation*}
where 
\[
P_n^m(t)=(1-t^2)^{\frac{m}{2}}\frac{{\rm d}^m}{{\rm d}t^m}P_n(t),\quad 1\leq t\leq 1, \, 0\leq m\leq n,
\]
are the associated Legendre functions and $P_n$ is the Legendre polynomial of degree $n$. Define a sequence of rescaled spherical harmonics of order $n$:
\begin{equation*}\label{rxy}
X_n^m(\theta,\varphi)=\frac{1}{R}Y_n^m(\theta,\varphi).
\end{equation*}
It is clear to note that $\{X_n^m(\theta,\varphi): |m|\leq n ,~
n=0,1,2,...\}$ forms a complete orthonormal system in $L^2(\Gamma_R)$.

For a smooth scalar function $w$ defined on $\Gamma_R$, let 
\begin{equation*}\label{sgrad}
\bold{\nabla}_\Gamma w=\frac{\partial w}{\partial\theta}\boldsymbol e_{\theta}+\frac{1}{\sin\theta}\frac{\partial w}{\partial\varphi}\boldsymbol e_{\varphi}
\end{equation*}
be the tangential gradient on $\Gamma_R$. It follows from \cite[Theorem 6.25]{ck-98} that 
\begin{equation*}
\boldsymbol U_n^m(\theta,\varphi)= \frac{1}{\sqrt{n(n+1)}}\bold{\nabla}_\Gamma X_n^m(\theta,\varphi),\quad \boldsymbol V_n^m(\theta,\varphi)=\boldsymbol e_{\rho}\times \boldsymbol U_n^m
\end{equation*}
for $m=-n, ..., n, n=1,2,...$ form a complete orthonormal system in $\boldsymbol{L}_t^2(\Gamma_R)=\{\boldsymbol w\in L^2(\Gamma_R)^3: \boldsymbol w\cdot\boldsymbol e_\rho=0\}$. For the convenience of notation, define 
\[
\boldsymbol U_0^0=0, \quad \boldsymbol V_0^0=0.
\]

\section{Some basic identities}\label{appendixoperators}

In this section, we list some basic identities in the spherical coordinates which are frequently used in the paper. The proofs are based on straightforward calculations. 

Let $f(\rho)$ be a smooth function. It can be verified that the gradient operator satisfies 
\begin{equation}\label{gradspherical}
\nabla\left(f(\rho) X_n^m\right)= f'(\rho) X_n^m\boldsymbol e_{\rho}+ f(\rho) \frac{\sqrt{n(n+1)}}{\rho}\boldsymbol U_n^m,
\end{equation}
the curl operator satisfies 
\begin{equation}\label{curlspherical}
\left\{
\begin{aligned}
& \nabla\times\left(f(\rho)\boldsymbol U_n^m\right)=\frac{1}{\rho}\frac{\partial}{\partial\rho}\left(\rho f(\rho)\right) \boldsymbol V_n^m,\\
& \nabla\times\left(f(\rho)\boldsymbol V_n^m\right)=-\frac{1}{\rho}\frac{\partial}{\partial\rho}\left(\rho f(\rho)\right)\boldsymbol U_n^m -\frac{\sqrt{n(n+1)}}{\rho}f(\rho) X_n^m\boldsymbol e_{\rho},\\
& \nabla\times\left(f(\rho) X_n^m \boldsymbol e_{\rho}\right)=-\frac{\sqrt{n(n+1)}}{\rho}f(\rho)\boldsymbol V_n^m,
\end{aligned}
\right. 
\end{equation}
the divergence operator satisfies
\begin{equation}\label{divspherical}
\left\{
\begin{aligned}
&\nabla\cdot\left(f(\rho)\boldsymbol U_n^{m}\right)=-f(\rho)\sqrt{(n+1)n}\,\frac{1}{\rho}X_{n}^{m},\\
&\nabla\cdot\left(f(\rho)\boldsymbol V_n^m\right)=0,\\
&\nabla\cdot\left(f(\rho)X_n^{m}\boldsymbol e_{\rho}\right)=\frac{1}{\rho^2}\frac{\partial}{\partial \rho}\left(\rho^2 f(\rho)\right)X_n^m,
\end{aligned}
\right.
\end{equation}
the Laplacian operator satisfies
\begin{align}\label{lapspherical}
\Delta f(\rho) &= \sum\limits_{n\in\mathbb{N}}\sum\limits_{|m|\leq n}\frac{1}{\rho^2}\frac{\partial}{\partial \rho}
\left(\rho^2 \frac{\partial}{\partial \rho} f_n^m(\rho)\right)X_n^m
+\frac{1}{\rho^2} f_n^m(\rho)\Delta_\Gamma X_n^m \notag\\
&= \sum\limits_{n\in\mathbb{N}}\sum\limits_{|m|\leq n}\left[\frac{2}{\rho}f_n^{m'}(\rho)
+f_n^{m''}(\rho)-\frac{n(n+1)}{\rho^2} f_n^m(\rho)\right]X_n^m,
\end{align}
and the double curl operator satisfies
\begin{equation}\label{curlcurlspherical}
\left\{
\begin{aligned}
& \nabla\times\left(\nabla\times\left(f(\rho)\boldsymbol U_n^m\right)\right)= -\frac{1}{\rho}\frac{\partial^2}{\partial \rho^2}\left(\rho f(\rho)\right)\boldsymbol U_n^m -\frac{\sqrt{(n(n+1))}}{\rho^2}\frac{\partial}{\partial\rho}\left(\rho f(\rho)\right)X_n^m \boldsymbol e_{\rho},\\
& \nabla\times\left(\nabla\times\left(f(\rho) \boldsymbol V_n^m\right)\right)= \left[-\frac{1}{\rho}\frac{\partial^2}{\partial \rho^2}\left(\rho f(\rho)\right)+\frac{n(n+1)}{\rho^2} f(\rho)\right]\boldsymbol V_n^m, \\
& \nabla\times\left(\nabla\times\left(f(\rho) X_n^m \boldsymbol e_{\rho}\right)\right)= \frac{\sqrt{n(n+1)}}{\rho}\frac{\partial}{\partial\rho} f(\rho) \boldsymbol U_n^m  +\frac{n(n+1)}{\rho^2} f(\rho) X_n^m \boldsymbol e_{\rho},
\end{aligned}
\right.
\end{equation}
where $\Delta_\Gamma$ is the Laplace--Beltrami operator on $\Gamma_R$ and is defined by 
\[
 \Delta_\Gamma=\frac{1}{\sin\theta}\frac{\partial}{\partial\theta}\left(\sin\theta\frac{\partial}{\partial\theta}\right)+\frac{1}{\sin^2\theta}\frac{\partial^2}{\partial\varphi}. 
\]

\section{Function spaces}\label{appendixfunction}

Denote by $L^2(\Omega)$ the square integrable functions on $\Omega$. Let $\boldsymbol L^2(\Omega)=L^2(\Omega)^3$ be equipped with the inner product and norm given as follows: 
\[
 (\boldsymbol{u}, \boldsymbol{v})=\int_\Omega
    \boldsymbol{u}\cdot\bar{\boldsymbol{v}}\,{\rm d}\boldsymbol{x},\quad
 \|\boldsymbol{u}\|_{\boldsymbol L^2(\Omega)}=(\boldsymbol{u},
\boldsymbol{u})^{1/2}.
\]
Let $H^1(\Omega)$ the standard Sobolev space with the norm given by
\[
 \|u\|_{H^1(\Omega)}=\left(\int_\Omega
|u(\boldsymbol{x})|^2+|\nabla u(\boldsymbol x)|^2\,{\rm
d}\boldsymbol{x}\right)^{1/2}. 
\]
Define $\boldsymbol H^1_{\partial D}(\Omega)=H^1_{\partial D}(\Omega)^3$, where $H^1_{\partial D}(\Omega):=\{u\in
H^1(\Omega): u=0~\text{on}~\partial D\}$. 

Denote by $H^s(\Gamma_R)$ the standard trace functional space which is equipped with the norm 
\[
\|u\|_{H^s(\Gamma_R)}=\left(\sum_{n=0}^\infty\sum_{m=-n}^n (1+n(n+1))^s |u_n^m|^2\right)^{1/2}, \quad u(R, \theta, \varphi)=\sum_{n=0}^\infty\sum_{m=-n}^n u_n^m X_n^m(\theta, \varphi). 
\]
Let $\boldsymbol H^s(\Gamma_R)=H^s(\Gamma_R)^3$ which is equipped with the norm 
\[
\|\boldsymbol u\|_{\boldsymbol H^s(\Gamma_R)}=\left(\sum_{n=0}^\infty\sum_{m=-n}^n (1+n(n+1))^s
|\boldsymbol u_n^m|^2\right)^{1/2},
\]
where $\boldsymbol u_n^m=(u_{1n}^m, u_{2n}^m, u_{3n}^m)$ and 
\[
\boldsymbol u(R, \theta, \varphi)=\sum_{n=0}^\infty\sum_{m=-n}^n u_{1n}^m\boldsymbol U_n^m(\theta, \varphi)+u_{2n}^m\boldsymbol V_n^m(\theta, \varphi)+u_{3n}^m X_n^m(\theta, \varphi)\boldsymbol e_{\rho}. 
\]
It can be verified that $\boldsymbol H^{-s}(\Gamma_R)$ is the dual space of $\boldsymbol H^s(\Gamma_R)$ with respect to the inner product
\[
\langle\boldsymbol{u}, \boldsymbol{v}\rangle_{\Gamma_R}=\int_{\Gamma_R} \boldsymbol{u}\cdot\bar{\boldsymbol{v}}{\rm
d}s=\sum_{n=0}^\infty\sum_{m=-n}^n u_{1n}^m\bar{v}_{1n}^m +u_{2n}^m\bar{v}_{2n}^m+u_{3n}^m \bar{v}_{3n}^m,
\]
where
\[
\boldsymbol v(R, \theta, \varphi)=\sum_{n=0}^\infty\sum_{m=-n}^n v_{1n}^m\boldsymbol U_n^m(\theta, \varphi)+v_{2n}^m\boldsymbol V_n^m(\theta, \varphi)+v_{3n}^m X_n^m(\theta, \varphi) \boldsymbol e_{\rho}.
\]

\section{The DtN operator}\label{appendixDtN}

The DtN operator $\mathscr{T}$ defined in \eqref{section2_TBC} was introduced in \cite{LY-ipi}.  For the self-contained purpose, we summary the results here.

Let $\boldsymbol u$ admit the Fourier expansion
\[
\boldsymbol{u}(\rho, \theta, \varphi)=\sum\limits_{n=0}^{\infty}\sum\limits_{m=-n}^n u_{1n}^m(\rho) \boldsymbol U_n^m(\theta, \varphi)+u_{2n}^m(\rho)\boldsymbol V_n^m(\theta, \varphi)+u_{3n}^m(\rho) X_n^m(\theta, \varphi)\boldsymbol e_{\rho}.
\]
Consider the Helmholtz decomposition 
\[
\boldsymbol{u}=\nabla\phi+\nabla\times\boldsymbol{\psi},\quad \nabla\cdot\boldsymbol{\psi}=0,
\]
where the potential functions $\phi$ and $\boldsymbol\psi$ have the Fourier expansions
\begin{align}
\phi(\rho, \theta, \varphi)&=\sum\limits_{n=0}^{\infty}\sum\limits_{m=-n}^n \frac{h_n^{(1)}(\kappa_p \rho)}{h_n^{(1)}(\kappa_p R)}\phi_n^m(R) X_n^m(\theta, \varphi),\label{AppDtN_phi}\\
\boldsymbol{\psi}(\rho, \theta, \varphi) &= \sum\limits_{n=0}^{\infty}\sum\limits_{m=-n}^n\frac{1}{\sqrt{n(n+1)}}\frac{R}{\rho}\frac{h_n^{(1)}(\kappa_s \rho)+\kappa_s \rho h_n^{(1)'}(\kappa_s \rho)}{h_n^{(1)}(\kappa_s R)}\psi_{3n}^m(R)\boldsymbol U_n^m(\theta, \varphi) \notag\\
&\quad -\frac{h_n^{(1)}(\kappa_s \rho)}{h_n^{(1)}(\kappa_s R)}\psi_{2n}^m(R)\boldsymbol V_n^m(\theta, \varphi)
+\left(\frac{R}{\rho}\frac{h_n^{(1)}(\kappa_s \rho)}{h_n^{(1)}(\kappa_s R)}\right)\psi_{3n}^m(R) 
X_n^m(\theta, \varphi)\boldsymbol e_{\rho}. \label{AppDtN_psi}
\end{align}
Here $h_n^{(1)}$ is the spherical Hankel function of the first kind with order $n$. 

Substituting \eqref{AppDtN_phi} and \eqref{AppDtN_psi} into the Helmholtz decomposition yields 
\begin{align}
	\begin{bmatrix}
		u_{1n}^m(R) \\ u_{2n}^m(R) \\ u_{3n}^m(R)
	\end{bmatrix}&=\frac{1}{R}
	\begin{bmatrix}
		\sqrt{n(n+1)} & -1-z_n^{(1)}(\kappa_s R) & 0\\
		0 & 0 & -\frac{\kappa_s^2 R^2}{\sqrt{n(n+1)}}\\
		z_n^{(1)}(\kappa_p R) & -\sqrt{n(n+1)} & 0
	\end{bmatrix}
	\begin{bmatrix}
		\phi_{n}^m(R) \\ \psi_{2n}^m(R) \\ \psi_{3n}^m(R)
	\end{bmatrix} \notag\\
	&= K_n(R)\begin{bmatrix}
		\phi_{n}^m(R) \\ \psi_{2n}^m(R) \\ \psi_{3n}^m(R)
	\end{bmatrix},  \label{DtN_Kn}
\end{align}
where $z_n^{(1)}(t)=t h_n^{(1)'}(t)/h_n^{(1)}(t)$. It follows from a simple calculation that the inverse of $K_n(R)$ is 
\begin{equation}\label{DtN_Kninverse}
	K_n^{-1}(R)=
	\frac{R}{\Lambda_n(R)}
	\begin{bmatrix}
		-\sqrt{n(n+1)} & 0 & 1+z_n^{(1)}(\kappa_s R) \\
		-z_n^{(1)}(\kappa_p R) & 0 & \sqrt{n(n+1)}\\
		0 & -\sqrt{n(n+1)} \Lambda_n(R)/(\kappa_s^2 R^2) & 0
	\end{bmatrix},
\end{equation}
where
\begin{equation}\label{DtN_Lambda}
\Lambda_n(R)=z_n^{(1)}(\kappa_p R)\left(1+z_n^{(1)}(\kappa_s R)\right)-n(n+1).
\end{equation}
It is shown in \cite{LY-ipi} that
\[
\Im \Lambda_n(R)=\Re z_n^{(1)}(\kappa_p R)\Im z_n^{(1)}(\kappa_s R)+\left(1+\Re z_n^{(1)}(\kappa_s R)\right)\Im z_n^{(1)}(\kappa_p R)<0,
\]
which guarantees the existence of the inverse for $K_n(R)$.

By the Helmholtz decomposition and \eqref{AppDtN_phi}--\eqref{AppDtN_psi}, the boundary operator $\mathscr{D}$ can be expressed in terms of the potential functions $\phi$ and $\boldsymbol{\psi}$ on $\Gamma_R$; it can be further written in terms of $\boldsymbol{u}$ on $\Gamma_R$ by \eqref{DtN_Kn}--\eqref{DtN_Kninverse}, which deduce the DtN operator $\mathscr{T}$. Explicitly, 
it is shown in \cite{LY-ipi} that the entries of the matrix $M_n(R)$ defined in \eqref{section2_TBC} are
\begin{align*}
& M_{n, 11}(R)=-\frac{\mu}{R}\left(1+\frac{z_n^{(1)}(\kappa_p R)}{\Lambda_n(R)}\kappa_s^2 R^2\right),\quad
M_{n, 13}(R)=\sqrt{n(n+1)}\frac{\mu}{R}\left(1+\frac{\kappa_s^2 R^2}{\Lambda_n(R)}\right), \\
& M_{n, 22}(R)=\frac{\mu}{R}z_n^{(1)}(\kappa_s R),\quad M_{n, 31}(R)=\sqrt{n(n+1)}\frac{\mu}{R}\left[1+\frac{\omega^2 R^2}{\mu\Lambda_n(R)}\right],\quad M_{n, 32}(R)=0,\\
& M_{n, 33}(R)=-\frac{\mu}{R}\left[2+\frac{\omega^2 R^2}{\mu\Lambda_n(R)}\left(1+z_n^{(1)}(\kappa_s R)\right)\right],
\quad M_{n, 12}(R)=M_{n, 21}(R)=M_{n, 23}(R)=0.
\end{align*}

Let $\hat{M}_n=-\frac{1}{2}\left(M_n+M_n^*\right)$, where $M_n^*$ is the adjoint of $M_n$.  It is also shown in \cite{LY-ipi} that for a fixed $R>0$ and a sufficiently large $n$, $\hat{M}_n$ is positive definite and its entries satisfy the estimate
\begin{equation}\label{Mnestimate}
|M_{n, ij}(R)|\lesssim n, \quad i,j=1,2,3.
\end{equation}

Introduce the DtN operators $\mathscr T_1$ and $\mathscr T_2$ for the potentials $\phi$ and $\boldsymbol\psi$ on $\Gamma_R$, respectively. In \cite{bzh-2020}, it is shown that 
\begin{equation}\label{DtN_T1}
\partial_{\rho} \phi=\mathscr{T}_1 \phi, \quad \mathscr{T}_1 \phi=\sum\limits_{n=0}^{\infty}\sum\limits_{m=-n}^n z_n^{(1)}(\kappa_p R) \phi_n^m(R) X_n^m(\theta, \varphi).
\end{equation}
In \cite{bz-2020}, it is shown that 
\begin{eqnarray*}
\left(\nabla\times\boldsymbol{\psi}\right)\times\boldsymbol e_{\rho}=-{\rm i}\kappa_s \mathscr{T}_2 \boldsymbol{\psi}_{\Gamma_R}, \quad\boldsymbol{\psi}_{\Gamma_R}=-\boldsymbol e_{\rho}\times\left(\boldsymbol e_{\rho}\times \boldsymbol{\psi}\right),
\end{eqnarray*}
where 
\begin{equation}\label{DtN_T2}
\mathscr{T}_2 \boldsymbol{\psi}=\sum\limits_{n=0}^{\infty}\sum\limits_{m=-n}^n\frac{{\rm i}\kappa_s R}{1+z_n^{(1)}(\kappa_s R)} \psi_{1n}^m(R)\boldsymbol U_n^m(\theta, \varphi)+\frac{1+z_n^{(1)}(\kappa_s R)}{{\rm i}\kappa_s R} \psi_{2n}^m(R)\boldsymbol V_n^m(\theta, \varphi).
\end{equation}

\section{Bessel functions}

This section is concerned with some properties of the Bessel functions, which are needed in the proof of Theorem \ref{Mainthm}. 

The proof of Lemma \ref{Bessel1} may be found in \cite{RS-15-arma}.

\begin{lemma}\label{Bessel1}
Let $\nu=n-1/2$, where $n$ is an arbitrary integer. For $0<z<\nu$, it holds that
\begin{equation*}\label{Yn1}
Y_{\nu}(z)=-\left(\frac{2}{\pi}\right)^{1/2}\frac{1}{\left(\nu^2-z^2\right)^{1/4}} e^{\eta(z, \nu)} \left[1+r(z, \nu)\right],
\end{equation*}
where
\begin{equation*}\label{eta}
\eta(z, \nu)=\nu\log\left(\frac{\nu}{z}+\left[\left(\frac{\nu}{z}\right)^2-1\right]^{1/2}\right)-\left(\nu^2-z^2\right)^{1/2}
\end{equation*}
and
\[
\left|r(z, \nu)\right|\leq e^{\tilde{g}_2} \tilde{g}_2, \quad g_2=\frac{\nu-z}{\nu^{1/3}},\quad\tilde{g}_2=\frac{2}{3\, g_2^{3/2}}.
\]
\end{lemma}

As a consequence, the results in Lemma \ref{Bessel1} can be used to estimate $Y_{\nu-1}(z)/Y_{\nu}$.

\begin{lemma}\label{Bessel2}
For a fixed $z\in\mathbb R^+$ and sufficiently large $\nu$, it holds that
\[
\frac{z}{2\nu}-\frac{1}{6}\frac{z}{\nu^2}+O\left(\frac{1}{n^3}\right)\leq \frac{Y_{\nu-1}(z)}{Y_{\nu}(z)}\leq \frac{z}{2\nu}+\frac{7}{6}\frac{z}{\nu^2}+O\left(\frac{1}{\nu^3}\right).
\]
\end{lemma}

\begin{proof}
It follows from Lemma \ref{Bessel1} that
\begin{equation}\label{Thm1_1}
\frac{Y_{\nu-1}(z)}{Y_{\nu}(z)}=\frac{\left(\nu^2-z^2\right)^{1/4}}{\left[(\nu-1)^2-z^2\right]^{1/4}}\frac{1+R(z, \nu-1)}{1+R(z, \nu)}e^{\eta(z, \nu-1)-\eta(z, \nu)}:=I_1 I_2 I_3.
\end{equation}

A direct calculation shows that 
\begin{equation}\label{I1est}
I_1=\frac{\left(\nu^2-z^2\right)^{1/4}}{\left[(\nu-1)^2-z^2\right]^{1/4}}= 1+\frac{1}{2\nu}+O\left(\frac{1}{\nu^2}\right)
\end{equation}
and
\begin{align*}
& \tilde{g}_2(z, \nu)=\frac{2}{3}\left(\frac{\nu^{1/3}}{\nu-z}\right)^{3/2}=\frac{2}{3\nu}+O\left(\frac{1}{\nu^2}\right),\\
& e^{\tilde{g}_2(z, \nu)} \tilde{g}_2(z, \nu) = \frac{2}{3\nu}+O\left(\frac{1}{\nu^2}\right),\\
& e^{\tilde{g}_2(z, \nu-1)} \tilde{g}_2(z, \nu-1)= \frac{2}{3\nu}+O\left(\frac{1}{\nu^3}\right). 
\end{align*}
Substituting the above estimates into $I_2$, we have 
\begin{align*}
& I_2 = \frac{1+R(z, \nu-1)}{1+R(z, \nu)}\leq \frac{1+\tilde{g}_2(z, \nu-1)e^{\tilde{g}_2(z, \nu-1)}}{1-\tilde{g}_2(z, \nu)e^{\tilde{g}_2(z, \nu)}}=1+\frac{4}{3}\frac{1}{\nu}+O\left(\frac{1}{\nu^2}\right),\\
& I_2 = \frac{1+R(z, \nu-1)}{1+R(z, \nu)}\geq \frac{1-\tilde{g}_2(z, \nu-1)e^{\tilde{g}_2(z, \nu-1)}}{1+\tilde{g}_2(z, \nu)e^{\tilde{g}_2(z, \nu)}}=1-\frac{4}{3\nu}+O\left(\frac{1}{\nu^2}\right),
\end{align*}
which give that 
\begin{equation}\label{I2est}
1-\frac{4}{3\nu}+O\left(\frac{1}{\nu^2}\right)\leq I_2\leq 1+\frac{4}{3\nu}+O\left(\frac{1}{\nu^2}\right). 
\end{equation}

It follows from a simple calculation that
\begin{align*}
&\eta(z, \nu-1)-\eta(z, \nu) =\left(\nu-1\right)\log\left(\frac{\nu-1}{z}+\left[\left(\frac{\nu-1}{z}\right)^2-1\right]^{1/2}\right)\\
& \quad-\nu\log\left(\frac{\nu}{z}+\left[\left(\frac{\nu}{z}\right)^2-1\right]^{1/2}\right)
+\left(\nu^2-z^2\right)^{1/2}-\left((\nu-1)^2-z^2\right)^{1/2}\\
&:= A+B
\end{align*}
It is easy to check that
\begin{equation}\label{estB}
B = \left(\nu^2-z^2\right)^{1/2}-\left((\nu-1)^2-z^2\right)^{1/2}= 1+\frac{z^2}{2\nu^2}+O\left(\frac{1}{\nu^3}\right)
\end{equation}
and
\begin{align*}
&\log\left(\frac{\nu}{z}+\left[\left(\frac{\nu}{z}\right)^2-1\right]^{1/2}\right)= \log\left(\frac{2\nu}{z}\right)-\frac{1}{4}\left(\frac{z}{\nu}\right)^2+O\left(\frac{1}{\nu^3}\right),\\
&\log\left(\frac{\nu-1}{z}+\left[\left(\frac{\nu-1}{z}\right)^2-1\right]^{1/2}\right)=\log\left(\frac{2\nu}{z}\right)-\frac{1}{\nu}-\left(\frac{1}{2}+\frac{z^2}{4}\right)\frac{1}{\nu^2}+O\left(\frac{1}{\nu^3}\right),
\end{align*}
which lead to 
\begin{align}\label{estA}
A &= \left(\nu-1\right)\log\left(\frac{\nu-1}{z}+\left[\left(\frac{\nu-1}{z}\right)^2-1\right]^{1/2}\right)-\nu\log\left(\frac{\nu}{z}+\left[\left(\frac{\nu}{z}\right)^2-1\right]^{1/2}\right) \notag\\
&= -\log\left(\frac{2\nu}{z}\right)-1+\frac{1}{2\nu}+O\left(\frac{1}{\nu^2}\right). 
\end{align}
Combining \eqref{estA} and \eqref{estB}, we obtain 
\begin{equation}\label{I3est}
I_3= \exp\left(\eta(z, \nu-1)-\eta(z, \nu)\right)=\exp\left(A+B\right)=\frac{z}{2\nu}+\frac{z}{4}\frac{1}{\nu^2}+O\left(\frac{1}{\nu^3}\right). 
\end{equation}

Substituting \eqref{I1est} and \eqref{I3est} into \eqref{Thm1_1} gives 
\begin{eqnarray*}
\frac{Y_{\nu-1}(z)}{Y_{\nu}(z)} = I_1 I_2 I_3 = \left[\frac{z}{2\nu}+\frac{z}{2}\frac{1}{\nu^2}++O\left(\frac{1}{\nu^3}\right)\right] I_2.
\end{eqnarray*}
We have from \eqref{I2est} that 
\begin{align*}
\frac{Y_{\nu-1}(z)}{Y_{\nu}(z)} &\leq\left[\frac{z}{2\nu}+\frac{z}{2}\frac{1}{\nu^2}++O\left(\frac{1}{\nu^3}\right)\right]\left[1+\frac{4}{3\nu}+O\left(\frac{1}{\nu^2}\right)\right]\\
&= \frac{z}{2\nu}+\frac{7z}{6}\frac{1}{\nu^2}+O\left(\frac{1}{\nu^3}\right)
\end{align*}
and
\begin{align*}
\frac{Y_{\nu-1}(z)}{Y_{\nu}(z)} &\geq \left[\frac{z}{2\nu}+\frac{z}{2}\frac{1}{\nu^2}++O\left(\frac{1}{\nu^3}\right)\right]
\left[1-\frac{4}{3\nu} +O\left(\frac{1}{\nu^2}\right)\right]\\
&= \frac{z}{2\nu}-\frac{z}{6}\frac{1}{\nu^2}+O\left(\frac{1}{\nu^3}\right),
\end{align*}
which completes the proof.
\end{proof}

\begin{lemma}\label{Bessel3}
Let $G_{\nu}(z)=\frac{z Y'_{\nu}(z)}{Y_{\nu}(z)}$. The following estimate holds for sufficiently large $\nu$:
\[
-\nu+\frac{z^2}{2\nu}-\frac{1}{6}\frac{z^2}{\nu^2}+O\left(\frac{1}{\nu^3}\right)\leq G_{\nu}(z)\leq -\nu+\frac{z^2}{2\nu}+\frac{7}{6}\frac{z^2}{\nu^2}+O\left(\frac{1}{\nu^3}\right)<0. 
\]
\end{lemma}

\begin{proof}
Recall the following recurrence relation for the Bessel function (cf. \cite{W-95}):
\[
Y'_{\nu}(z)=Y_{\nu-1}(z)-\frac{\nu}{z}Y_{\nu}(z),
\]
which gives 
\begin{eqnarray*}
G_{\nu}(z) = \frac{z Y'_{\nu}(z)}{Y_{\nu}(z)}= z\left[\frac{Y_{\nu-1}(z)}{Y_{\nu}(z)}-\frac{\nu}{z}\right].
\end{eqnarray*}
By Theorem \ref{Bessel2}, we obtain for sufficiently large $\nu$ that 
\begin{eqnarray*}
G_{\nu}(z) \leq  z\left[ \frac{z}{2\nu}+\frac{7}{6}\frac{z}{\nu^2}+O\left(\frac{1}{\nu^3}\right)-\frac{\nu}{z}\right]
=-\nu+\frac{z^2}{2\nu}+\frac{7}{6}\frac{z^2}{\nu^2}+O\left(\frac{1}{\nu^3}\right)
\end{eqnarray*}
and
\begin{eqnarray*}
G_{\nu}(z) \geq  z \left[\frac{z}{2\nu}-\frac{z}{6}\frac{1}{\nu^2}+O\left(\frac{1}{\nu^3}\right)-\frac{\nu}{z}\right]
=-\nu+\frac{z^2}{2\nu}-\frac{1}{6}\frac{z^2}{\nu^2}+O\left(\frac{1}{\nu^3}\right),
\end{eqnarray*}
which completes the proof.
\end{proof}

\begin{lemma}\label{Bessel4}
Assume $R'<R$ and $\kappa_p<\kappa_s$. Then for sufficiently large positive $\nu$, it holds that
\[
\left|\frac{Y_{\nu}(\kappa_p R)}{Y_{\nu}(\kappa _p R')}-\frac{Y_{\nu}(\kappa_s R)}{Y_{\nu}(\kappa_s R')}\right|
\leq \frac{7}{3}\frac{\kappa_s(\kappa_s-\kappa_p) }{\nu}R\left(R-R'\right)\left(\frac{R'}{R}\right)^\nu.
\]
\end{lemma}

\begin{proof}
Let $F_{\nu}(z)=Y_{\nu}(zR)/Y_{\nu}(z R')$. By the mean value theorem, we have 
	\begin{align}
		F_{\nu}(\kappa_p)-F_{\nu}(\kappa_s) &=F'_{\nu}(\xi)(\kappa_p-\kappa_s) \notag\\
		&= \frac{R Y'_\nu(\xi R) Y_\nu(\xi R')- R' 
			Y_\nu(\xi R) Y'_\nu(\xi R')}{Y_\nu(\xi R')^2}(\kappa_p-\kappa_s)\notag\\
		&= \left[G_{\nu}(\xi R)-G_{\nu}(\xi R')\right]\frac{Y_\nu(\xi R)}{Y_\nu(\xi R')}\frac{\kappa_p-\kappa_s}{\xi}\notag\\
		&= G'_{\nu}(\eta)(R-R')(\kappa_p-\kappa_s)\frac{Y_\nu(\xi R)}{Y_\nu(\xi R')}, \label{Thm3_1}
	\end{align}
where $\eta$ satisfies $ \kappa_p R'<\eta<\kappa_s R $. A simple calculation yields 
\begin{eqnarray*}
G'_{\nu}(z) = \left(\frac{z Y'_{\nu}(z)}{Y_{\nu}(z)}\right)' = \frac{1}{z}\frac{zY'_{\nu}(z)Y_{\nu}(z)+z^2Y''_{\nu}(z)Y_{\nu}(z)- z^2Y'_{\nu}(z)^2}{Y^2_{\nu}(z)}.
\end{eqnarray*}
Using the identity 
\[
z^2 Y''_{\nu}(z)=-zY'_{\nu}(z)-(z^2-\nu^2)Y_{\nu}(z),
\]
we obtain 
\begin{align}\label{Gderiv}
G'_{\nu}(z) &= \frac{1}{z}\frac{zY'_{\nu}(z)Y_{\nu}(z)+\left[-zY'_{\nu}(z)-(z^2-\nu^2)Y_{\nu}(z)\right]Y_{\nu}(z)- z^2Y'_{\nu}(z)^2}{Y^2_{\nu}(z)} \notag\\
&= -\frac{z^2-\nu^2}{z}-\frac{1}{z}G^2_{\nu}(z). 
\end{align}

It is shown in Lemma \ref{Bessel3} that $G_{\nu}(z)$ is always negative for sufficiently large $\nu$. Hence we have 
\begin{eqnarray*}
G^2_{\nu}(z) \leq \left[-\nu+\frac{z^2}{2\nu}-\frac{1}{6}\frac{z^2}{\nu^2}+O\left(\frac{1}{\nu^3}\right)\right]^2
= \nu^2-z^2+\frac{1}{3}\frac{z^2}{\nu}+O\left(\frac{1}{\nu^2}\right)
\end{eqnarray*}
and
\begin{eqnarray*}
G^2_{\nu}(z) \geq \left[ -\nu+\frac{z^2}{2\nu}+\frac{7}{6}\frac{z^2}{\nu^2}+O\left(\frac{1}{n^3}\right)\right]^2
=\nu^2-z^2-\frac{7}{3}\frac{z^2}{\nu}+O\left(\frac{1}{\nu^2}\right). 
\end{eqnarray*}
Substituting the above estimates into \eqref{Gderiv} leads to 
\begin{align*}
G'_{\nu}(z)\leq -z+\frac{\nu^2}{z}-\frac{1}{z}\left[\nu^2-z^2-\frac{7}{3}\frac{z^2}{\nu}+O\left(\frac{1}{\nu^2}\right)\right]=\frac{7}{3}\frac{z}{\nu}+O\left(\frac{1}{\nu^2}\right)
\end{align*}
and
\begin{align*}
G'_{\nu}(z)\geq -z+\frac{\nu^2}{z}-\frac{1}{z}\left[\nu^2-z^2+\frac{1}{3}\frac{z^2}{\nu}+O\left(\frac{1}{\nu^2}\right)\right]
=-\frac{1}{3}\frac{z}{\nu}+O\left(\frac{1}{\nu^2}\right). 
\end{align*}

Combining the above estimates, we get from \eqref{Thm3_1} that 
\begin{align*}
G'_{\nu}(\eta)(R-R')(\kappa_p-\kappa_s)\leq -\frac{1}{3}\frac{\eta}{\nu}(R-R')(\kappa_p-\kappa_s)\leq \frac{1}{3}\frac{\kappa_s R}{\nu}(R-R')(\kappa_s-\kappa_p)
\end{align*}
and
\begin{align*}
G'_{\nu}(\eta)(R-R')(\kappa_p-\kappa_s)\geq \frac{7}{3}\frac{\eta}{\nu}(R-R')(\kappa_p-\kappa_s)
\geq -\frac{7}{3}\frac{\kappa_s R}{\nu}(R-R')(\kappa_s-\kappa_p),
\end{align*}
which shows 
\[
|G'_{\nu}(\eta)(R-R')(\kappa_p-\kappa_s)| \leq \frac{7}{3}\frac{\kappa_s R}{\nu}(R-R')(\kappa_s-\kappa_p).
\]
The proof is completed by noting that
\[
\frac{Y_{\nu}(zR)}{Y_{\nu}(z R')}\sim \left(\frac{R'}{R}\right)^{\nu}.
\]
\end{proof}

\begin{lemma}\label{Bessel5}
	Let $0<\kappa_p<\kappa_s$ and $0<\hat R<R$. For sufficiently large $n$, the
	following estimate holds for $j=1, 2$:
	\begin{equation*}
		\left|\frac{h_{n}^{(j)}(\kappa_p R)}{h_{n}^{(j)}(\kappa_p  R')}
		-\frac{h_{n}^{(j)}(\kappa_s R)}{h_{n}^{(j)}(\kappa_s R')}\right|
		\leq\frac{14}{3}\frac{\kappa_s(\kappa_s-\kappa_p) }{n}R\left(R-R'\right)\left(\frac{R'}{R}\right)^{n+1},
	\end{equation*}
	where $h_n^{(1)}$ and $h_n^{(2)}$ are the spherical Hankel functions of the first and
	second kind with order $n$, respectively.  
\end{lemma} 

\begin{proof}
Since the spherical Hankel functions of the first and second kind are complex conjugate to each other, it is only needed to prove  the result for the first kind of spherical Hankel functions. 

Recalling the relationship between the spherical Bessel functions and the Bessel functions 
\[
h_n^{(j)}(z)=\sqrt{\frac{\pi}{2z}}H_{n+\frac{1}{2}}^{(j)}, \quad y_n(z)=\sqrt{\frac{\pi}{2z}}Y_{n+\frac{1}{2}}, \quad j_n(z)=\sqrt{\frac{\pi}{2z}}J_{n+\frac{1}{2}}, 
\]
we have from \cite{W-95} that
\[
j_{n}(z)\sim \frac{1}{z}\sqrt{\frac{1}{2e}}\left(\frac{ez}{2n+1}\right)^{n+1}, \quad y_n(z)\sim -\frac{1}{z}\sqrt{\frac{2}{e}}\left(\frac{ez}{2n+1}\right)^{-n}.
\]
Let $S_n(z)=\frac{j_n(z)}{y_n(z)}$. Then it is easy to check that
\[
S_n(z)\sim -\frac{1}{2}\left(\frac{ez}{2n+1}\right)^{2n+1}.
\]

It is shown in \cite{LY-2Dobstacle} that
\begin{align*}
& \left|\frac{h_{n}^{(1)}(\kappa_pR)}{h_{n}^{(1)}(\kappa_p R')}-\frac{h_{n}^{(1)}(\kappa_s R)}{h_{n}^{(1)}(\kappa_s R')}\right|
\leq \left|\frac{y_n(\kappa_p R)}{y_n(\kappa_p R')}-\frac{y_n(\kappa_s R)}{y_n(\kappa_s R')}\right|+\left|\frac{y_n(\kappa_p R)}{y_n(\kappa_p R')}\frac{S_n(\kappa_p R')}{1-{\rm i}S_n(\kappa_p R')}\right| \\
&\quad +\left|\frac{y_n(\kappa_p R)}{y_n(\kappa_p R')}\frac{S_n(\kappa_p R)}{1-{\rm i}S_n(\kappa_p R')}\right| +\left|\frac{y_n(\kappa_s R)}{y_n(\kappa_s R')}\frac{S_n(\kappa_s R')}{1-{\rm i}S_n(\kappa_s R')}\right|+\left|\frac{y_n(\kappa_s R)}{y_n(\kappa_s R')}\frac{S_n(\kappa_s R)}{1-{\rm i}S_n(\kappa_s R')}\right|. 
\end{align*}
A simple computation yields that
\begin{align*}
\left|\frac{S_n(z R)}{1-{\rm i}S_n(z R')}\right|\leq \left(\frac{ezR}{2n+1}\right)^{2n+1},\quad\left|\frac{S_n(z R')}{1-{\rm i}S_n(z R')}\right|\leq \left(\frac{ez R'}{2n+1}\right)^{2n+1}. 
\end{align*}
Combining the above estimates, we have for $R>\hat R$ and $\kappa_s>\kappa_p$ that 
\begin{align*}
& \left|\frac{y_n(\kappa_p R)}{y_n(\kappa_p R')}\frac{S_n(\kappa_p R')}{1-{\rm i}S_n(\kappa_p R')}\right| +\left|\frac{y_n(\kappa_p R)}{y_n(\kappa_p R')}\frac{S_n(\kappa_p R)}{1-{\rm i}S_n(\kappa_p R')}\right| +\left|\frac{y_n(\kappa_s R)}{y_n(\kappa_s R')}\frac{S_n(\kappa_s R')}{1-{\rm i}S_n(\kappa_s R')}\right|\\
&\quad +\left|\frac{y_n(\kappa_s R)}{y_n(\kappa_s R')}\frac{S_n(\kappa_s R)}{1-{\rm i}S_n(\kappa_s R')}\right| \leq 2\left(\frac{e\kappa_s R}{2n+1}\right)^{2n+1}\left(\left|\frac{y_n(\kappa_p R)}{y_n(\kappa_p R')}\right| +\left|\frac{y_n(\kappa_s R)}{y_n(\kappa_s R')}\right|\right).
\end{align*}

By Lemma \ref{Bessel4}, we have
\begin{align*}
\left|\frac{y_{n}(\kappa_p R)}{y_{n}(\kappa _p R')}-\frac{y_{n}(\kappa_s R)}{y_{n}(\kappa_s R')}\right|
&=\sqrt{\frac{R'}{R}}\left|\frac{Y_{\nu+1}(\kappa_p R)}{Y_{\nu+1}(\kappa _p R')}-\frac{Y_{\nu+1}(\kappa_s R)}{Y_{\nu+1}(\kappa_s R')}\right|\\
&\leq \frac{7}{3}\frac{\kappa_s(\kappa_s-\kappa_p) }{n+1}R\left(R-R'\right)\left(\frac{R'}{R}\right)^{n+1}.
\end{align*}
Therefore,
\begin{align*}
\left|\frac{h_{n}^{(1)}(\kappa_p R)}{h_{n}^{(1)}(\kappa_p R')}-\frac{h_{n}^{(1)}(\kappa_s R)}{h_{n}^{(1)}(\kappa_s R')}\right|
&\leq\frac{7}{3}\frac{\kappa_s(\kappa_s-\kappa_p) }{n+1}R\left(R-R'\right)\left(\frac{R'}{R}\right)^{n+1}\\
&\quad +2\left(\frac{e\kappa_s R}{2n+1}\right)^{2n+1}
\left(\left|\frac{y_n(\kappa_p R)}{y_n(\kappa_p R')}\right| 
+\left|\frac{y_n(\kappa_s R)}{y_n(\kappa_s R')}\right|\right)\\
&\leq \frac{14}{3}\frac{\kappa_s(\kappa_s-\kappa_p) }{n}R\left(R-R'\right)\left(\frac{R'}{R}\right)^{n+1},
\end{align*}
which completes the proof. 
\end{proof}

\begin{lemma}\label{LambdaN}
Given $\kappa_p, \kappa_s,$ and $R$, it holds for sufficiently large $n$ that 
\begin{eqnarray}\label{Lambdan}
\Lambda_n(R)= -\left(\kappa_p^2+\kappa_s^2\right)\frac{R^2}{2}+O\left(\frac{1}{n}\right).
\end{eqnarray}
\end{lemma}

\begin{proof}
It is clear to note for sufficiently large $n$ that 
\[
z^{(1)}_n(t)=\frac{th_n^{(1)'}(t)}{h_n^{(1)}(t)}\sim \frac{ty_n^{(1)'}(t)}{y_n^{(1)}(t)}. 
\]
We have from the recursive equation of the spherical Bessel functions that 
\begin{align*}
\frac{ty_n^{(1)'}(t)}{y_n^{(1)}(t)} = t\left[\frac{y_{n-1}(t)}{y_n(t)}-\frac{n+1}{t}\right]=  t\left[\frac{Y_{\nu}(t)}{Y_{\nu+1}(t)}-\frac{\nu+1}{t}\right]-\frac{1}{2}=G_{\nu+1}(t)-\frac{1}{2}. 
\end{align*}
By Lemma \ref{Bessel3}, we get 
\begin{align*}
z^{(1)}_{n}(t)&=G_{\nu+1}(t)-\frac{1}{2}\leq -n-1+\frac{t^2}{2n}+\frac{11}{12}\frac{t^2}{n^2}+O\left(\frac{1}{n^3}\right),\\
z^{(1)}_{n}(t)&=G_{\nu+1}(t)-\frac{1}{2}\geq -n-1+\frac{t^2}{2n}-\frac{5}{12}\frac{t^2}{n^2}+O\left(\frac{1}{n^3}\right),
\end{align*}
which give
\begin{equation}\label{Thm2_bessel}
-n-1+\frac{t^2}{2n}-\frac{5}{12}\frac{t^2}{n^2}+O\left(\frac{1}{n^3}\right)\leq z^{(1)}_{n}(t)\leq -n-1+\frac{t^2}{2n}+\frac{11}{12}\frac{t^2}{n^2}+O\left(\frac{1}{n^3}\right).
\end{equation}
The proof is completed by substituting \eqref{Thm2_bessel} into \eqref{DtN_Lambda}.
\end{proof}

\section{The dual problems}\label{appendixdual}

The dual problems in \eqref{section4_HelmholtzD} for the Helmholtz equation and the Maxwell equation are considered in 
\cite{bzh-2020} and \cite{bz-2020}, respectively. For the self-contained purpose, we summarize the related results here.

\begin{lemma}\label{step21}
Let $\hat{\zeta}=\frac{1}{\lambda+2\mu}\zeta$. The boundary value problem
\begin{equation*}
\begin{cases}
\Delta g+\kappa_p^2 g=-\hat{\zeta} \quad &{\rm in} ~ B_R\setminus \overline{B_{R'}},\\
\partial_{\rho}g=\mathscr{T}_1^*g \quad & {\rm on} ~ \partial B_R,\\
g=g \quad  &{\rm on} ~ \partial B_{R'}
\end{cases}
\end{equation*}
has a unique solution given by 
\begin{equation}\label{3Ddual_solutiong}
g_n^m(\rho)=S^p_n(\rho)g_n^m(R')+\frac{{\rm i}\kappa_p}{2}\int_{R'}^{\rho} t^2 W^p_n(\rho, t)\hat{\zeta}_n^m(t){\rm d}t
+\frac{{\rm i}\kappa_p}{2}\int_{R'}^R t^2 S^p_n(t) W^p_n(R', \rho)\hat{\zeta}_n^m(t){\rm d}t,
\end{equation}
where $g_n^m$ and $\hat{\zeta}_n^m$ are the Fourier coefficients of $g$ and $\hat{\zeta}$ with respect to the basis functions $X_n^m$, and 
\[
S^p_n(\rho)=\frac{h_n^{(2)}(\kappa_p \rho)}{h^{(2)}_n(\kappa_p R')},\quad W^p_n(\rho, t)={\rm det}\,
\begin{bmatrix}
h_n^{(1)}(\kappa_p \rho) & h_n^{(2)}(\kappa_p \rho)\\
h_n^{(1)}(\kappa_p t) & h_n^{(2)}(\kappa_p t)
\end{bmatrix}.
\]
\end{lemma}

Taking $\rho=R$ in \eqref{3Ddual_solutiong} yields 
\begin{equation}\label{appendix_gR}
g_n^m(R)=S_n^p(R) g_n^m(R')+\frac{{\rm i}\kappa_p}{2}\int_{R'}^R t^2 S_n^p(R) W^p_n(R', t)\hat{\zeta}_n^m(t){\rm d}t.
\end{equation}
Taking the derivative of \eqref{3Ddual_solutiong} and estimating it at $\rho=R'$, we get 
\begin{equation}\label{appendix_g'R'}
g_n^{m'}(R')=\frac{1}{R'}z_n^{(2)}(\kappa_p R') g_n^m(R')+\frac{2\kappa_p}{\pi R'}\int_{R'}^R t^2 S_n^p(t)
\hat{\zeta}_n^m(t){\rm d}t.
\end{equation}

\begin{lemma}\label{step22}
Let $\hat{Z}_{2n}^m=\frac{1}{\mu}Z_{2n}^{m}$, where $Z_{2n}^m$ are the Fourier coefficients of $Z$ under the basis $V_n^m$. The two-point boundary value problem
\begin{equation*}
\begin{cases}
q_{2n}^{m''}(\rho)+\frac{2}{\rho}q_{2n}^{m'}(\rho)+\left(\kappa_s^2-\frac{n(n+1)}{\rho^2}\right)q_{2n}^m(\rho)=-\hat{Z}_{2n}^{m}(\rho),\quad &\rho\in(R', R),\\
q_{2n}^{m'}(R)-\frac{z_n^{(2)}(\kappa_s R)}{R} q_{2n}^m(R)=0, \quad &\rho=R,\\
q_{2n}^{m}(R')=q_{2n}^{m}(R'),\quad & \rho=R'
\end{cases}
\end{equation*}
has a unique solution given by 
\begin{equation}\label{3Ddual_solutionq2}
q_{2n}^m(\rho)=S^s_n(\rho)q_{2n}^m(R')+\frac{{\rm i}\kappa_s}{2}\int_{R'}^{\rho} t^2 W^s_n(\rho, t)\hat{Z}_{2n}^m(t){\rm d}t
+\frac{{\rm i}\kappa_s}{2}\int_{R'}^R t^2 S^s_n(t) W^s_n(R', \rho)\hat{Z}_{2n}^m(t){\rm d}t,
\end{equation}
where
\[
S^s_n(\rho)=\frac{h_n^{(2)}(\kappa_s \rho)}{h^{(2)}_n(\kappa_s R')},\quad W^s_n(\rho, t)={\rm det}\,
\begin{bmatrix}
h_n^{(1)}(\kappa_s \rho) & h_n^{(2)}(\kappa_s \rho)\\
h_n^{(1)}(\kappa_s t) & h_n^{(2)}(\kappa_s t)
\end{bmatrix}.
\]
\end{lemma}

Taking $\rho=R$ in \eqref{3Ddual_solutionq2} gives
\begin{equation}\label{appendix_q2R}
q_{2n}^m(R)=S_n^s(R) q_{2n}^m(R')+\frac{{\rm i}\kappa_s}{2}\int_{R'}^R t^2 S_n^s(R) W^s_n(R', t)\hat{Z}_{2n}^m(t){\rm d}t.
\end{equation}
Taking the derivative of \eqref{3Ddual_solutionq2} and estimating it at $\rho=R'$, we obtain 
\begin{equation}\label{solutiondq2}
q_{2n}^{m'}(R')=\frac{1}{R'}z_n^{(2)}(\kappa_s R') q_{2n}^m(R')+\frac{2\kappa_s}{\pi R'}\int_{R'}^R t^2 S_n^s(t)\hat{Z}_{2n}^m(t){\rm d}t.
\end{equation}

\begin{lemma}\label{step23}
Let $v_{n}^m(\rho)=\rho q_{3n}^{m}(\rho)$. Then $v_{n}^{m}$ satisfies the two-point boundary value problem
\begin{equation*}
\begin{cases}
v_{n}^{m''}(\rho)+\frac{2}{\rho}v_{n}^{m'}(\rho)+\left(\kappa_s^2-\frac{n(n+1)}{\rho^2}\right)v_{n}^m(\rho)
=-\beta_{n}^{m}(\rho),\quad &\rho\in(R', R),\\
v_{n}^{m'}(R)-\frac{z_n^{(2)}(\kappa_s R)}{R} v_{n}^m(R)=0, \quad &\rho=R,\\
v_{n}^{m}(R')=v_{n}^{m}(R'), \quad &\rho=R',
\end{cases}
\end{equation*}
where $\beta_n^{m}(\rho)=\frac{1}{\mu}\rho Z_{3n}^m(\rho)$. Moreover, $v_n^m(R)$ and $v_n^{m'}(R')$ are given by 
\begin{align*}
v_{n}^m(R)&=S_n^s(R) v_{n}^m(R')+\frac{{\rm i}\kappa_s}{2}\int_{R'}^R t^2 S_n^s(R) W^s_n(R', t)\beta_{n}^m(t){\rm d}t,\\
v_{n}^{m'}(R')&=\frac{1}{R'}z_n^{(2)}(\kappa_s R') v_{n}^m(R')+\frac{2\kappa_s}{\pi R'}\int_{R'}^R t^2 S_n^s(t)
\beta_{n}^m(t){\rm d}t.
\end{align*}
\end{lemma}

Let $\hat{Z}_{3n}^m=\frac{1}{\mu}Z_{3n}^m$. We have from Lemma \ref{step23} that 
\begin{align}\label{appendix_q3R}
q_{3n}^m(R)&=\frac{R'}{R}S_n^s(R) q_{3n}^m(R')+\frac{{\rm i}\kappa_s}{2R}\int_{R'}^R t^3 S_n^s(R) W^s_n(R', t)
\hat{Z}_{3n}^m(t){\rm d}t,\\
\label{solutiondq3}
q_{3n}^{m'}(R')&=\frac{1}{R'}\left[z_n^{(2)}(\kappa_s R')-1\right] q_{3n}^m(R')+\frac{2\kappa_s}{\pi (R')^2}\int_{R'}^R t^3 S_n^s(t)\hat{Z}_{3n}^m(t){\rm d}t.
\end{align}

\end{document}